\documentclass[12pt,leqno]{article} 
\usepackage{amsmath,amssymb,amsthm,amsfonts, indentfirst}
\usepackage{enumerate,color,bm,graphicx,here}
\makeatletter
\def\@cite#1#2{[{{\bfseries #1}\if@tempswa , #2\fi}]}
\renewcommand{\section}{%
\@startsection{section}{1}{\z@}
{0.5truecm plus -1ex minus -.2ex}%
{1.0ex plus .2ex}{\bfseries\large}}
\def\@seccntformat#1{\csname the#1\endcsname.\ }
\makeatother

\numberwithin{equation}{section} 
\theoremstyle{theorem}
\newtheorem{thm}{Theorem}[section]

\newtheorem{lem}[thm]{Lemma}

\theoremstyle{definition}

\newtheorem{remark}{Remark}[section]

\newtheorem*{prth1.1}{Proof of Theorem 1.1 (existence part)}
\newtheorem*{prth1.1c}{Proof of Theorem 1.1 (continued)}
\newtheorem*{prth1.2}{Proof of Theorem 1.2}
\newtheorem*{prth1.3}{Proof of Theorem 1.3}

\newcommand{\ep}{\varepsilon}
\newcommand{\pa}{\partial}

\newcommand{\tmax}{T_{\rm max}}

\begin{document}
\footnote[0]
    {2010{\it Mathematics Subject Classification}\/. 
    Primary: 35B44; Secondary: 35K55, 92C17.
    }
\begin{center}
    \Large{{\bf Boundedness and finite-time blow-up in a quasilinear parabolic--elliptic chemotaxis system\\ with logistic source and nonlinear production}}
\end{center}
\vspace{5pt}
\begin{center}
    Yuya Tanaka
   \footnote[0]{
    E-mail: 
    {\tt yuya.tns.6308@gmail.com}
    }\\
    \vspace{12pt}
    Department of Mathematics, 
    Tokyo University of Science\\
    1-3, Kagurazaka, Shinjuku-ku, 
    Tokyo 162-8601, Japan\\
    \vspace{2pt}
\end{center}
\begin{center}    
    \small \today
\end{center}

\vspace{2pt}
\newenvironment{summary}
{\vspace{.5\baselineskip}\begin{list}{}{%
     \setlength{\baselineskip}{0.85\baselineskip}
     \setlength{\topsep}{0pt}
     \setlength{\leftmargin}{12mm}
     \setlength{\rightmargin}{12mm}
     \setlength{\listparindent}{0mm}
     \setlength{\itemindent}{\listparindent}
     \setlength{\parsep}{0pt}
     \item\relax}}{\end{list}\vspace{.5\baselineskip}}
\begin{summary}
{\footnotesize {\bf Abstract.}
This paper deals with the quasilinear parabolic--elliptic chemotaxis system with logistic source and nonlinear production,
\begin{equation*}
  \begin{cases}
    u_t=\nabla \cdot (D(u) \nabla u) - \nabla \cdot (S(u)\nabla v) 
          + \lambda u - \mu u^{\kappa},
    & x\in\Omega,\ t>0, 
  \\
    0=\Delta v - \overline{M_f}(t) + f(u),
    & x\in\Omega,\ t>0,
  \end{cases}
\end{equation*}
where $\lambda>0$, $\mu>0$, $\kappa>1$ and $\overline{M_f}(t):=\frac{1}{|\Omega|}\int_{\Omega} f(u(x,t))\,dx$, and $D$, $S$ and $f$ are functions generalizing the prototypes
  \begin{align*}
    D(u)=(u+1)^{m-1},\quad 
    S(u)=u(u+1)^{\alpha-1}\quad\mbox{and}\quad
    f(u)=u^\ell
  \end{align*}
with $m\in\mathbb{R}$, $\alpha>0$ and $\ell>0$.
In the case $m=\alpha=\ell=1$, Fuest (NoDEA Nonlinear Differential Equations Appl.; 2021; 28; 16) obtained conditions for $\kappa$ such that solutions blow up in finite time.
However, in the above system boundedness and finite-time blow-up of solutions have been not yet established.
This paper gives boundedness and finite-time blow-up under some conditions for $m$, $\alpha$, $\kappa$ and $\ell$.
} 
\end{summary}
\vspace{10pt}

\newpage
\section{Introduction}

In this paper we consider the following quasilinear parabolic--elliptic chemotaxis system with logistic source and nonlinear production:
\begin{equation}\label{JL}
  \begin{cases}
    u_t=\nabla \cdot (D(u) \nabla u) - \nabla \cdot (S(u)\nabla v) 
          + \lambda u - \mu u^{\kappa},
    & x\in\Omega,\ t>0, 
  \\
    0=\Delta v - \overline{M_f}(t) + f(u),
    & x\in\Omega,\ t>0,
  \\
    \nabla u \cdot \nu=\nabla v \cdot \nu=0,
    & x\in \pa\Omega,\ t>0,
  \\
    u(x,0)=u_0(x),
    & x\in\Omega,
  \end{cases}
\end{equation}
where $\Omega \subset \mathbb{R}^n\ (n\ge1)$ is a bounded domain with smooth boundary $\pa\Omega$;
$\lambda>0$, $\mu>0$ and $\kappa>1$; 
$D,S \in C^2([0,\infty))$ and $D(0)>0$;
$f\in \bigcup_{\beta\in(0,1)} C_{\rm loc}^\beta([0,\infty)) \cap C^1((0,\infty))$;
  \[
    \overline{M_f}(t):=\frac{1}{|\Omega|}\int_{\Omega} f(u(x,t))\,dx;
  \]
$\nu$ is the outward normal vector to $\pa\Omega$;
$u_0 \in \bigcup_{\beta\in(0,1)} C^\beta(\overline{\Omega})$ is nonnegative.

\medskip

The system \eqref{JL} describes a motion of cellular slime molds with chemotaxis, and the unknown function $u=u(x,t)$ donates the density of cells and the unknown function $v=v(x,t)$ represents the concentration of the chemical substance at place $x \in \Omega$ and time $t>0$.
This system is one of many types of the Keller--Segel system
\begin{equation}\label{KSsimple}
  \begin{cases}
    u_t=\Delta u - \nabla \cdot (u\nabla v),
    & x\in\Omega,\ t>0, 
  \\
    v_t=\Delta v - v + u,
    & x\in\Omega,\ t>0,
  \end{cases}
\end{equation}
which was proposed by Keller and Segel \cite{K-S}.
A number of variations of the original system \eqref{KSsimple} and related results for blow-up 
(in the radial setting) and boundedness are introduced in \cite{B-B-T-W,H-P,L-W}: 

\begin{itemize}
\item We first focus on the quasilinear Keller--Segel system,
\begin{equation*}
  \begin{cases}
    u_t=\nabla \cdot (D(u) \nabla u) - \nabla \cdot (S(u)\nabla v), 
    & x\in\Omega,\ t>0, 
  \\
    \tau v_t=\Delta v - v + f(u),
    & x\in\Omega,\ t>0,
  \end{cases}
\end{equation*}
where $\tau\in\{0,1\}$.
When $f(u)=u$, in the parabolic--parabolic setting $(\tau=1)$, 
Tao and Winkler \cite{T-W} showed that solutions are global and bounded under the conditions that $\frac{S(u)}{D(u)}\le cu^q$ with $q<\frac{2}{n}$ and $c>0$ and 
that $\Omega$ is a convex domain; Ishida et al. \cite{I-S-Y} removed the convexity of $\Omega$;
whereas Winkler \cite{W_2010_quasilinear} proved that solutions blow up in either finite or infinite time when $\frac{S(u)}{D(u)}\ge cu^q$ with $q>\frac{2}{n}$ and $c>0$;
In the parabolic--elliptic setting $(\tau=0)$, Lankeit \cite{L-2020} proved that
solutions exist globally and are bounded in the case $q<\frac{2}{n}$ and that unbounded solutions are constructed in the case $q>\frac{2}{n}$.
When $\tau=1$ and $D(u)=1$, $S(u)=u$ and $f(u)=u^\ell$ with $\ell>0$, Liu and Tao \cite{Liu-Tao} established global existence and boundedness under the condition that $0<\ell<\frac{2}{n}$;
Moreover, in the case that $D(u)=(u+1)^{m-1}$ and $S(u)=u(1+u)^{\alpha-1}$ 
with $m\in \mathbb{R}$ and $\alpha \in \mathbb{R}$, it was obtained that solutions are global and bounded under the condition $\alpha-m+\max\left\{\ell,\frac{1}{n}\right\}<\frac{2}{n}$ in \cite{TVY_2021}.

\item We next review the quasilinear Keller--Segel system with logistic source,
\begin{equation*}
  \begin{cases}
    u_t=\nabla \cdot (D(u) \nabla u) - \nabla \cdot (S(u)\nabla v) 
          + \lambda u - \mu u^{\kappa},
    & x\in\Omega,\ t>0, 
  \\
    \tau v_t=\Delta v - v + f(u),
    & x\in\Omega,\ t>0,
  \end{cases}
\end{equation*}
where $\lambda>0$, $\mu>0$, $\kappa>1$ and $\tau\in\{0,1\}$.
As to this system, blow-up phenomena are suppressed 
when $\kappa\ge2$ and $f(u)=u$.
Indeed, in the parabolic--parabolic setting ($\tau=1$), when $D(u)=1$ and $S(u)=u$, 
Winkler \cite{W-2010} derived that solutions exist globally and are bounded if $\mu>0$ is so large and $\kappa=2$;
When $D(u)= (u+1)^{m-1}$ and $S(u)=u(u+1)^{\alpha-1}$ with $m\in \mathbb{R}$ and $\alpha\in\mathbb{R}$, global existence and boundedness were obtained if $\lambda=\mu=1$, $\kappa=2$ and $0<\alpha-m+1<\frac{4}{4+n}$ by Zheng \cite{Z-2017}.
In the parabolic--elliptic setting $(\tau=0)$, when $D(u)=1$ and $S(u)=u$, 
Tello and Winker \cite{Tello-W-2007} showed that solutions exist globally and are bounded in the cases that $\kappa=2$ and $\mu>\max\left\{0,\frac{n-2}{n}\right\}$ and that $\kappa>2$ and $\mu>0$;
When $D(u)=u^{m-1}$ and $S(u)=u^\alpha$ for all $u\ge1$ with $m\ge1$ and $\alpha>0$, Zheng \cite{Z-2015} proved global existence and boundedness in the cases that $\kappa>1$ and $\alpha+1<\max\left\{m+\frac{2}{n},\kappa\right\}$ and that $\kappa>1$, 
$\alpha+1=\kappa$ and $\mu>\mu_0$ for some $\mu_0>0$.
On the other hands, in the parabolic--elliptic setting, it is known that blow-up occurs under the some conditions for $\kappa>1$ when $f(u)=u$.
When $D(u)=1$ and $S(u)=u$, Winkler \cite{W-2018} presented that if $1<\kappa<\frac{7}{6}\ (n\in\{3,4\})$ and $1<\kappa<1+\frac{1}{2(n-1)}\ (n\ge5)$, then solutions blow up in finite time;
Similar blow-up results were obtained in the case that  $D(u)=(u+1)^{m-1}$ and $S(u)=u(u+1)^{\alpha-1}$ with $m\ge1$ and $\alpha>0$ (see \cite{B-F-L, T_2021_KS_logistic, Tanaka-Y_2020}).

\item We turn our eyes into the quasilinear parabolic--elliptic chemotaxis system
\begin{equation*}
  \begin{cases}
    u_t=\nabla \cdot (D(u) \nabla u) - \nabla \cdot (S(u)\nabla v),
    & x\in\Omega,\ t>0, 
  \\
    0=\Delta v - \overline{M_f}(t) + f(u),
    & x\in\Omega,\ t>0.
  \end{cases}
\end{equation*}
A simplification of this system was introduced by J\"{o}ger and Luckhaus \cite{J-L}. 
When $D(u)=(u+1)^{m-1}$ with $m\in\mathbb{R}$, $S(u)=u$ and $f(u)=u$, Cie\'{s}lak and Winkler \cite{C-W} derived global existence and boundedness in the case $2-m<\frac{2}{n}$ and finite-time blow-up in the case $2-m>\frac{2}{n}$; 
When $D(u)=(u+1)^{m-1}$ and $S(u)=u(u+1)^{\alpha-1}$ with $m\le1$ and 
$\alpha\in\mathbb{R}$ as well as $f(u)=u$, Winkler and Djie \cite{W-D} proved that solutions are global and bounded if $\alpha-m+1<\frac{2}{n}$, whereas finite-time blow-up occurs if $\alpha-m+1>\frac{2}{n}$;
When $D(u)=1$, $S(u)=u$ and $f(u)=u^\ell$ with $\ell>0$, Winkler \cite{W-2018_nonlinear}
obtained that solutions exist globally and remain bounded in the case $\ell<\frac{2}{n}$ and that there exist solutions which are unbounded in finite time in the case $\ell>\frac{2}{n}$;
When $D(u)=(u+1)^{m-1}$, $S(u)=u$ and $f(u)=u^\ell$ with $m\in\mathbb{R}$ and $\ell>0$, global existence and boundedness were established if $\ell-m+1<\frac{2}{n}$ by Li \cite{Li_2019}. 
Moreover, in \cite{Li_2019} it was asserted that finite-time blow-up occurs under the condition that $\ell-m+1>\frac{2}{n}$. 
However, this condition should be repaired because from assumptions of \cite[Lemma 3.5]{Li_2019} we can obtain the condition that
  \begin{equation}\label{Licondi}
    \ell-(m-1)_+>\frac{2}{n},\quad\mbox{where}\quad (m-1)_+:=\max\{0,m-1\};
  \end{equation}
When $D(u)=1$, $S(u)=u(u+1)^{\alpha-1}$ and $f(u)=u^\ell$ with $\alpha>0$ and $\ell>0$,
Wang and Li \cite{Wang-Li} derived the critical value $\alpha+\ell-1=\frac{2}{n}$.

\item In the system \eqref{JL}, 
when $D(u)=1$, $S(u)=u$ and $f(u)=u$, Winkler \cite{W-2011} showed that
if $1<\kappa<\frac{3}{2}+\frac{1}{2n-2}\ (n\ge5)$, then there exists a solution blowing up in finite time; Moreover, a similar blow-up result was obtained in the case that $D(u)=(u+1)^{m-1}$ with $m\ge1$ in \cite{B-F-L};
Furthermore, Fuest \cite{F_2021_optimal} showed that solutions blow up in finite time under the conditions that $1<\kappa<\min\left\{2,\frac{n}{2}\right\}$ and $\mu>0\ (n\ge3)$ and that $\kappa=2$ and $\mu\in\left(0,\frac{n-4}{n}\right)\ (n\ge5)$;
In the two dimensional setting and $\kappa=2$, global existence and boundedness were established when $\int_\Omega u_0<8\pi$, whereas finite-time blow-up occurs when $\int_\Omega u_0<m_0$ with $m_0>8\pi$ in \cite{F-2020}.
\end{itemize}

In summary,
in \cite{B-F-L, F-2020,F_2021_optimal,W-2011}, blow-up results were derived in the chemotaxis system with logistic source and {\it linear} production.
However, 
boundedness and blow-up results were not obtained in the quasilinear chemotaxis system with logistic source and {\it nonlinear} production (when $D(u)=1$ and $S(u)=u$, recently, 
Yi et al. \cite{YMXD} derived the blow-up result under the condition that $\ell+1>\kappa\left(1+\frac{2}{n}\right)$).

\medskip

Our aim of this paper is to present conditions that solutions of \eqref{JL} are bounded or blow up.
Before we state the main results, we give conditions for the functions $D$, $S$ and $f$ as follows:
  \begin{align}
    \label{Dclass}
    D &\in C^2([0,\infty))\ \mbox{is positive};
    \\ \label{Sclass}
    S &\in C^2([0,\infty))\ \mbox{is nonnegative and nondecreasing};
    \\ \label{fclass}
    f &\in \bigcup_{\beta\in(0,1)} C_{\rm loc}^\beta([0,\infty)) \cap C^1((0,\infty))
    \ \mbox{is nonnegative and  nondecreasing}.
  \end{align}

We now state the main theorems. The first one asserts boundedness of solutions.
%
%
\begin{thm}\label{thm1}
Let $\Omega \subset \mathbb{R}^n\ (n\ge1)$ be a smooth bounded domain, 
and let $\delta \in (0,1]$, $m\in\mathbb{R}$, $\alpha>0$, $\lambda>0$, $\mu>0$, $\kappa>1$ and $\ell>0$.
Assume that $u_0 \in \bigcup_{\beta\in(0,1)} C^\beta(\overline{\Omega})$ is nonnegative
and $D$, $S$ and $f$ satisfy \eqref{Dclass}, \eqref{Sclass} and \eqref{fclass} as well as 
  \begin{align}
    \label{DSineq1}
    D(\xi) \ge C_D (\xi + \delta)^{m-1},
  \quad
    S(\xi) \le C_S \xi(\xi + \delta)^{\alpha-1}
    \quad\mbox{for all}\ \xi\ge0
  \end{align}
and
  \begin{align}\label{fineq1}
    f(\xi) \le L \xi^\ell
    \quad\mbox{for all}\ \xi\ge0
  \end{align}
with $C_D>0$, $C_S>0$ and $L>0$.
Suppose that $m$, $\alpha$, $\mu$, $\kappa$ and $\ell$ fulfill that
  \begin{alignat}{2}
    \label{thm1condi1}
    \mbox{if}\ 
    \alpha+\ell&<\max\left\{m+\frac{2}{n},\kappa\right\},& &\quad\mbox{then}\quad \mu>0,
  \\ 
    \label{thm1condi2}
    \mbox{if}\      
    \alpha+\ell&=\kappa,& &\quad\mbox{then}\quad
    \mu>\frac{n(\alpha+\ell-m)-2}{2(\alpha-1)+n(\alpha+\ell-m)}C_SL.
  \end{alignat}
Then there exists an exactly one pair $(u,v)$ of functions
  \begin{align*}
    \begin{cases}
      u\in C^0(\overline{\Omega}\times[0,\infty))\cap C^{2,1}(\overline{\Omega}\times(0,\infty)),\\
      v\in \bigcap_{q>n}C^0([0,\infty);W^{1,q}(\Omega))\cap C^{2,0}(\overline{\Omega}\times(0,\infty))
    \end{cases}
  \end{align*}
which solves \eqref{JL} classically. Moreover, the solution $(u,v)$ is bounded in the sense that
there exists $C>0$ such that
  \[
    \|u(\cdot,t)\|_{L^\infty(\Omega)} \le C
    \quad\mbox{for all}\ t \in(0,\infty).
  \]
\end{thm}
%

We next state a result such that solutions blow up in finite time.
%
%
\begin{thm}\label{thm2}
Let $\Omega:=B_{R}(0) \subset \mathbb{R}^n\ (n\ge1)$ be a ball with some $R>0$, 
and let $\delta \in (0,1]$, $m\in\mathbb{R}$, $\alpha>0$, $\lambda>0$, $\mu>0$, $\kappa>1$ and $\ell>0$.
Assume that 
$D$, $S$ and $f$ satisfy \eqref{Dclass}, \eqref{Sclass} and \eqref{fclass} as well as 
  \begin{align}
    \label{DSineq2}
    D(\xi) \le C_D (\xi + \delta)^{m-1},
  \quad
    S(\xi) \ge C_S \xi(\xi + \delta)^{\alpha-1}
    \quad\mbox{for all}\ \xi\ge0
  \end{align}
and
  \begin{align}\label{fineq2}
    f(\xi) \ge L \xi^\ell
    \quad\mbox{for all}\ \xi\ge0
  \end{align}
with $C_D>0$, $C_S>0$ and $L>0$.
Suppose that
  \begin{align}
    \label{thm2condi1}
    &\mbox{if}\ m\ge0, \quad\mbox{then}\quad 
    \alpha+\ell > \max\left\{m+\frac{2}{n}\kappa,\kappa\right\},
  \\ \label{thm2condi2}
    &\mbox{if}\ m<0, \quad\mbox{then}\quad
    \alpha+\ell > \max\left\{\frac{2}{n}\kappa,\kappa\right\}.
  \end{align}
Then for all $M_0>0$ there exist $\ep_0\in(0,M_0)$ and $r_\star\in(0,R)$ with the 
following property\/{\rm :}
If
  \begin{align}\label{u0}
    u_0 \in \bigcup_{\beta\in(0,1)} C^\beta(\overline{\Omega})\ \mbox{is nonnegative, radially symmetric, nonincreasing with respect to $|x|$}
  \end{align}
and
  \begin{align}\label{u0mass}
    \int_\Omega u_0(x)\,dx=M_0
    \quad\mbox{and}\quad
    \int_{B_{r_\star}(0)}u_0(x)\,dx\ge M_0-\ep_0,
  \end{align}
then there exist $T^*\in(0,\infty)$ and an exactly one pair $(u,v)$ of functions
  \begin{align*}
    \begin{cases}
      u\in C^0(\overline{\Omega}\times[0,T^*))\cap C^{2,1}(\overline{\Omega}\times(0,T^*)),\\
      v\in \bigcap_{q>n}C^0([0,T^*);W^{1,q}(\Omega))\cap C^{2,0}(\overline{\Omega}\times(0,T^*))
    \end{cases}
  \end{align*}
which solves \eqref{JL} classically and blows up in the sense that
  \[
    \lim_{t\nearrow T^*}\|u(\cdot,t)\|_{L^{\infty}(\Omega)}=\infty.
  \]
\end{thm}
%
%
\begin{remark}
As to Theorem \ref{thm1}, letting $\kappa\to1$ implies that the condition \eqref{thm1condi1}
reduces the condition 
  \[
    \alpha+\ell<\max\left\{m+\frac{2}{n},1\right\},
  \]
which is a generalized condition such that solutions remain bounded 
in \cite{Li_2019, Wang-Li, W-2018_nonlinear}.
Also, as to Theorem \ref{thm2}, we see that the condition \eqref{thm2condi1} with $m=1$ and $\kappa\to1$ is a generalized condition such that solutions blow up in finite time in \cite{Wang-Li, W-2018_nonlinear}.
\end{remark}
%
%
\begin{remark}
When $\alpha=1$, letting $\kappa\to1$ entails from
 \eqref{thm2condi1} and \eqref{thm2condi2} that
  \begin{align}
    \label{remcondi1}
    &\mbox{if}\ m\ge0, \quad\mbox{then}\quad 
    \ell > \max\left\{m-1+\frac{2}{n},0\right\},
  \\ \label{remcondi2}
    &\mbox{if}\ m<0, \quad\mbox{then}\quad
    \ell > \max\left\{-1+\frac{2}{n},0\right\}.
  \end{align}
For instance, when $m\le1-\frac{2}{n}$, we obtain from \eqref{Licondi} that $\ell>\frac{2}{n}$, whereas we can observe from \eqref{remcondi1} and \eqref{remcondi2} that $\ell>\left\{\frac{2}{n}-1,0\right\}$.
Thus the conditions \eqref{remcondi1} and \eqref{remcondi2} improve the condition  in \cite{Li_2019}.

\begin{figure}[H]
\begin{minipage}{0.48\columnwidth}
\begin{center}
\scalebox{1.0}{\includegraphics{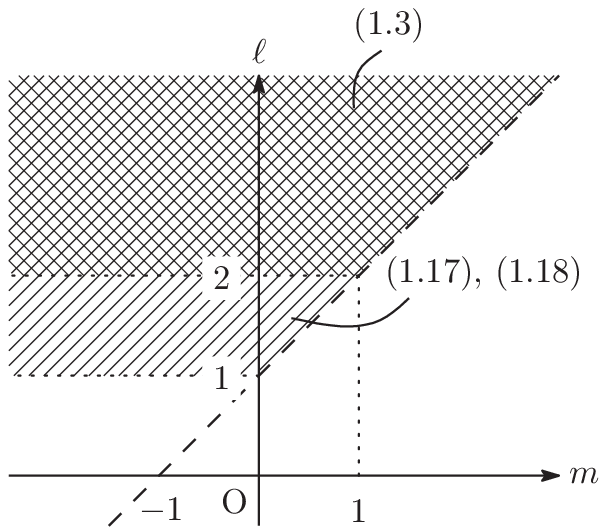}}
\caption{$n=1$, $\alpha=1$ and $\kappa\to1$}
\end{center}
\end{minipage}
\begin{minipage}{0.48\columnwidth}
\begin{center}
\scalebox{1.0}{\includegraphics{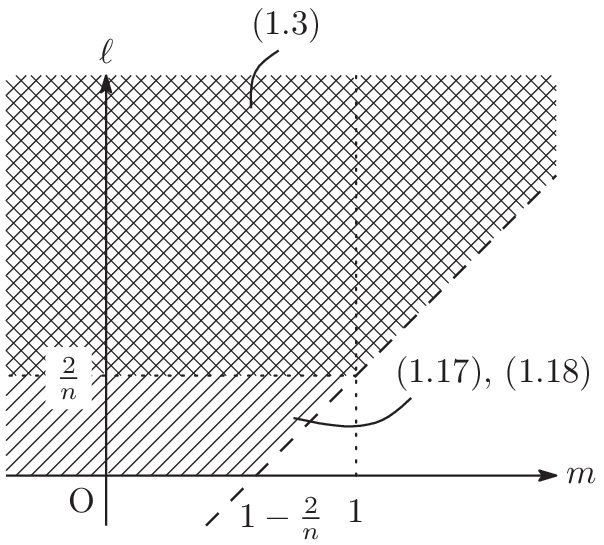}}
\vspace{-1.5mm}
\caption{$n\ge2$, $\alpha=1$ and $\kappa\to1$}
\end{center}
\end{minipage}
\end{figure}

\noindent
Moreover, in the case that $m=1$ and $\alpha=1$, we can establish that
  \begin{align}
    \label{lkcondi}
    1+\ell > \max\left\{1+\frac{2}{n}\kappa,\kappa\right\}.
  \end{align}
Because $\left(1+\frac{2}{n}\right)\kappa>\max\left\{1+\frac{2}{n}\kappa,\kappa\right\}$, we can make sure that the condition \eqref{lkcondi} is an improvement on the condition in \cite{YMXD}.

\begin{figure}[h]
\begin{minipage}{0.48\columnwidth}
\begin{center}
\scalebox{1.0}{\includegraphics{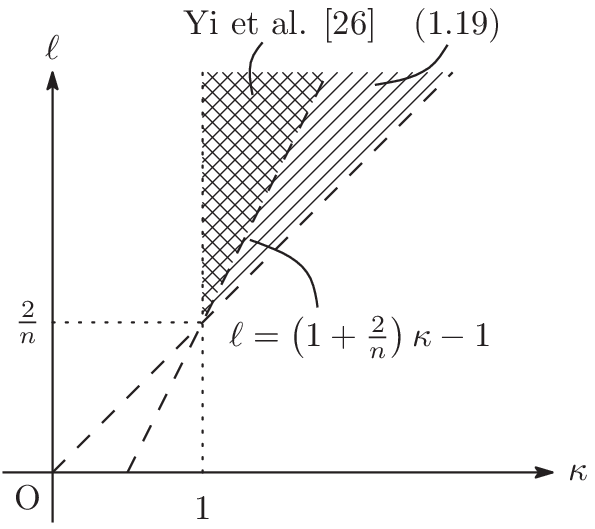}}
\caption{$n\in\{1,2\}$, $m=1$ and $\alpha=1$}
\end{center}
\end{minipage}
\begin{minipage}{0.48\columnwidth}
\begin{center}
\scalebox{1.0}{\includegraphics{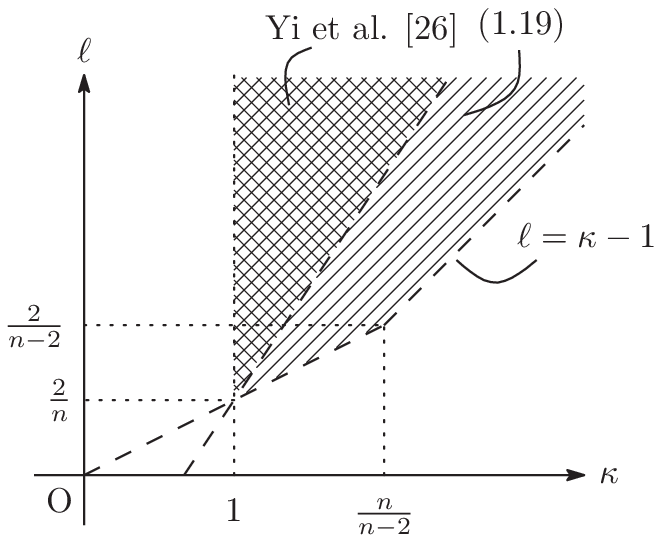}}
\vspace{-1.5mm}
\caption{$n\ge3$, $m=1$ and $\alpha=1$}
\end{center}
\end{minipage}
\end{figure}
\end{remark}

The proofs of Theorems \ref{thm1} and \ref{thm2} are based on those in \cite{W-2018_nonlinear}.
As to the proof of Theorem \ref{thm1}, our purpose is to establish an $L^p$-estimate for $u$. In order to obtain an $L^p$-estimate, we consider three cases.
With regard to the proof of Theorem \ref{thm2}, we first define the mass accumulation function
  \[
    w(s,t):=\int^{s^\frac{1}{n}}_{0} \rho^{n-1}u(\rho,t)\,d\rho
    \quad\mbox{for}\ s \in [0,R^n]\ \mbox{and}\ t \in [0,\tmax),
  \]
where $s:=r^n$ for $r\in[0,R]$, and transform the system \eqref{JL} to the parabolic equation
  \begin{align*}
    w_t &= n^2s^{2-\frac{2}{n}}D(nw_s)w_{ss}
            - \frac{1}{n} sS(nw_s)\overline{M_f}(t)
    \\ \notag
    &\quad\,
            + \frac{1}{n} S(nw_s)\int^{s}_{0}f(nw_s(\sigma,t))\,d\sigma
            +\lambda w
            - n^{\kappa-1}\mu \int^{s}_{0} w_s^\kappa(\sigma,t)\,d\sigma.
  \end{align*}
Next, we introduce the moment-type functional 
  \[
    \phi(t):=\int^{s_0}_{0} s^{-\gamma}(s_0-s)w(s,t)\,ds
  \]
and the functional
  \[
    \psi(t):=\int^{s_0}_{0} s^{1-\gamma}(s_0-s)w_s^{\alpha+\ell}(s,t)\,ds
  \]
with some $s_0\in(0,R^n)$ and $\gamma\in(-\infty,1)$.
By using the above functionals, we will derive nonlinear differential inequalities $\phi'\ge c_1\phi^{\alpha+\ell}-c_2$.
In order to attain this inequality, we apply the inequality $\psi\ge c_3\phi^{\alpha+\ell}$
(in \cite{YMXD} the inequality $\psi\ge c_4\phi^{\frac{1+\ell}{\kappa}}$ with some $c_4>0$ was obtained).
Moreover, in the case $m=0$, due to use the estimate $\log(a+\delta)\le \frac{1}{\ep}a^\ep+c_5$ for all $\ep>0$ with some $c_5>0$, we can improve the condition \eqref{Licondi} to the conditions \eqref{remcondi1} and \eqref{remcondi2}.

\medskip

This paper is organized as follows. In Section \ref{boundedness} we recall local existence and show Theorem \ref{thm1}.
In Section \ref{blow-up} we prove Theorem \ref{thm2} and give open problems.
%
\section{Boundedness}\label{boundedness}
%

In this section we derive global existence and boundedness in \eqref{JL}.
We first introduce a result on local existence of classical solutions to \eqref{JL}.
This lemma can be proved by a standard fixed point argument (see e.g., \cite{W-D}). 
%
%
\begin{lem}\label{localsol}
Let $\Omega \subset \mathbb{R}^n\ (n\ge1)$ be a smooth bounded domain, 
and let $\lambda>0$, $\mu>0$ and $\kappa>1$.
Assume that $u_0 \in \bigcup_{\beta\in(0,1)} C^\beta(\overline{\Omega})$ is nonnegative and $D$, $S$ and $f$ fulfill \eqref{Dclass}, \eqref{Sclass} and \eqref{fclass}.
Then there exist $\tmax \in (0,\infty]$ and a unique classical solution $(u,v)$ of \eqref{JL}
satisfying 
  \begin{align*}
    \begin{cases}
      u\in C^0(\overline{\Omega}\times[0,\tmax))\cap C^{2,1}(\overline{\Omega}\times(0,\tmax)),\\
      v\in \bigcap_{q>n}C^0([0,\tmax);W^{1,q}(\Omega))\cap C^{2,0}(\overline{\Omega}\times(0,\tmax)).
    \end{cases}
  \end{align*}
Moreover, $u\ge0$ in $\Omega \times (0,\tmax)$ and 
  \begin{align*}
    \mbox{if}\ \tmax<\infty, \quad \mbox{then}\quad
    \lim_{t\nearrow\tmax}\|u(\cdot,t)\|_{L^{\infty}(\Omega)}=\infty.
  \end{align*}
If $u_0$ is radially symmetric, 
then so are $u(\cdot,t)$ and $v(\cdot,t)$ for all $t \in (0,\tmax).$
\end{lem}
%
In the following we assume that $\Omega \subset \mathbb{R}^n\ (n\ge1)$ is a smooth bounded domain and $\delta \in (0,1]$, $m\in\mathbb{R}$, $\alpha>0$, $\lambda>0$, $\mu>0$, $\kappa>1$ and $\ell>0$. 
Also, we suppose that $D$, $S$ and $f$ satisfy \eqref{DSineq1} and \eqref{fineq1}.
Moreover, let $(u,v)$ be the solution of \eqref{JL} on $[0,\tmax)$ as in Lemma \ref{localsol}. 
We next recall the following lemma which is obtained from the first equation in \eqref{JL}.
%
%
\begin{lem}\label{L1}
The classical solution $u$ satisfies that
  \begin{align}\label{L1lem}
    \int_\Omega u(x,t)\,dx \le M_*:=\max\left\{\int_\Omega u_0(x)\,dx, \left(\frac{\lambda}{\mu}|\Omega|^{\kappa-1}\right)^\frac{1}{\kappa-1}\right\} \quad \mbox{for all}\ t \in (0,\tmax).
  \end{align}
\end{lem}
\begin{proof}
Integrating the first equation in \eqref{JL} and using H\"{o}lder's inequality, we have
\[
  \frac{d}{dt}\int_\Omega u\,dx \le \lambda \int_\Omega u\,dx - \mu|\Omega|^{1-\kappa}\left(\int_\Omega u\,dx\right)^\kappa
\]
for all $t \in (0,\tmax)$. By an ODE comparison argument we attain \eqref{L1lem}.
\end{proof}

In order to see global existence and boundedness of solutions, it is sufficient to make sure that for each nonnegative initial data $u_0 \in \bigcup_{\beta\in(0,1)} C^\beta(\overline{\Omega})$ and for any $p>1$ we can take $C=C(p)>0$ such that
\begin{align}\label{Lp}
  \int_\Omega u^p(x,t)\,dx \le C \quad \mbox{for all}\ t \in (0,\tmax).
\end{align}
In the following subsections we will prove \eqref{Lp} in three cases as follows:
\begin{itemize}
  \item Case $1$. $\alpha+\ell<m+\frac{2}{n}$ and $\mu>0$.
  \item Case $2$. $\alpha+\ell<\kappa$ and $\mu>0$.
  \item Case $3$. $\alpha+\ell=\kappa$ and $\mu>\frac{n(\alpha+\ell-m)-2}{2(\alpha-1)+n(\alpha+\ell-m)}C_SL$.
\end{itemize}

\subsection{Case $\mathbf{1}$. $\alpha+\ell<m+\frac{2}{n}$ and $\mu>0$.}

In this subsection we derive \eqref{Lp} 
under the condition that $\alpha+\ell<m+\frac{2}{n}$ and $\mu>0$. 
%
%
\begin{lem}\label{GBlem1}
Let $\mu>0$ and assume that $m\in\mathbb{R}$, $\alpha>0$ and $\ell>0$ satisfy 
\begin{align}\label{case1}
  \alpha+\ell<m+\frac{2}{n}.
\end{align}
Then for any $p>\max\left\{1,2-m,2-(\alpha+\ell),\frac{n}{2}(1-m)+\left(\frac{n}{2}-1\right)(\alpha+\ell-1)\right\}$ there is $C>0$ such that 
\begin{align}\label{lemLp-1}
  \int_\Omega u^p(x,t)\,dx \le C \quad \mbox{for all}\ t \in (0,\tmax).
\end{align}
\end{lem}
\begin{proof}
By virtue of the first equation in \eqref{JL} and 
$D(u) \ge C_D (u+\delta)^{m-1}$, we have
  \begin{align}\label{Lpcom}
    \frac{d}{dt} \int_\Omega (u+\delta)^p\,dx 
    &\le -p(p-1)C_D \int_\Omega (u+\delta)^{p+m-3}|
\nabla u|^2\,dx 
    \\ \notag
    &\quad\,
           +p(p-1) \int_\Omega (u+\delta)^{p-2}S(u) \nabla u \cdot \nabla v\,dx
    \\ \notag
    &\quad\,
           +p \lambda \int_\Omega u(u+\delta)^{p-1}\,dx 
           -p \mu \int_\Omega u^\kappa(u+\delta)^{p-1}\,dx
   \\ \notag
    &=   -\frac{4p(p-1)C_D}{(p+m-1)^2} \int_\Omega |\nabla (u+\delta)^{\frac{p+m-1}{2}}|^2\,dx 
    \\ \notag
    &\quad\,
           +p(p-1) \int_\Omega 
             \nabla \left(\int_{0}^{u} (\xi+\delta)^{p-2}S(\xi)\,d\xi\right) \cdot \nabla v\,dx
    \\ \notag
    &\quad\,
           +p \lambda \int_\Omega u(u+\delta)^{p-1}\,dx 
           -p \mu \int_\Omega u^\kappa(u+\delta)^{p-1}\,dx
    \\ \notag
    &=: I_1 + I_2 + I_3 + I_4
  \end{align}
for all $t \in (0,\tmax)$.
Noting from $S(\xi) \le C_S (\xi+\delta)^{\alpha}$ and $p>1-\alpha$ that
  \[
    \int_{0}^{u} (\xi+\delta)^{p-2}S(\xi)\,d\xi 
    \le C_S \int_{0}^{u} (\xi+\delta)^{p+\alpha-2}\,d\xi
    \le \frac{C_S}{p+\alpha-1} (u+\delta)^{p+\alpha-1},
  \]
from \eqref{fineq1} and the second equation in \eqref{JL} we can obtain
  \begin{align}\label{I2com}
    I_2&= -p(p-1) 
              \int_\Omega \left(\int_{0}^{u} (\xi+\delta)^{p-2}S(\xi)\,d\xi\right)\Delta v\,dx
    \\ \notag
       &\le \frac{p(p-1)C_S}{p+\alpha-1} \int_\Omega (u+\delta)^{p+\alpha-1}f(u)\,dx
    \\ \notag
       &\le \frac{p(p-1)C_SL}{p+\alpha-1} \int_\Omega (u+\delta)^{p+\alpha+\ell-1}\,dx
  \end{align}
for all $t \in (0,\tmax)$. As to $I_3$ and $I_4$, since we see from elementary calculations that there exists $\varepsilon>0$ so small such that
  \[
    (u+\delta)^{\kappa} \le (1+\ep)u^\kappa + C_\ep \delta,
  \]
where $C_\ep:=\left(\frac{\delta}{1-(1+\ep)^{-\frac{1}{\kappa-1}}}\right)^{\kappa-1}>0$,
we can observe
  \begin{align}\label{I3I4com}
    &I_3+I_4 
    \\ \notag
    &\quad\,
      \le p \lambda \int_\Omega u(u+\delta)^{p-1}\,dx 
                   -\frac{p\mu}{1+\ep} \int_\Omega (u+\delta)^{p+\kappa-1}\,dx
                   +\frac{p\mu C_\ep}{1+\ep} \int_\Omega \delta(u+\delta)^{p-1}\,dx
    \\ \notag
    &\quad\,
      \le \widetilde{C}_\ep \int_\Omega (u+\delta)^{p}\,dx 
           -\frac{p\mu}{1+\ep} \int_\Omega (u+\delta)^{p+\kappa-1}\,dx
  \end{align}
for all $t \in (0,\tmax)$, where $\widetilde{C}_\ep:=\max\left\{p\lambda, \frac{p\mu C_\ep}{1+\ep}\right\}>0$.
From \eqref{Lpcom}--\eqref{I3I4com} it follows that
  \begin{align}\label{comesti}
    &\frac{d}{dt} \int_\Omega (u+\delta)^p\,dx 
    \\ \notag
    &\quad\,
      \le -\frac{4p(p-1)C_D}{(p+m-1)^2} \int_\Omega |\nabla (u+\delta)^{\frac{p+m-1}{2}}|^2\,dx 
           +\frac{p(p-1)C_SL}{p+\alpha-1} \int_\Omega (u+\delta)^{p+\alpha+\ell-1}\,dx
    \\ \notag
    &\quad\,\quad\,
           +\widetilde{C}_\ep \int_\Omega (u+\delta)^{p}\,dx 
           -\frac{p\mu}{1+\ep} \int_\Omega (u+\delta)^{p+\kappa-1}\,dx
  \end{align}
for all $t \in (0,\tmax)$. 
Here, let 
  \[
    \theta:=\frac{\frac{p+m-1}{2}-\frac{p+m-1}{2(p+\alpha+\ell-1)}}{\frac{p+m-1}{2}+\frac{1}{n}-\frac{1}{2}}.
  \]
From $p>\max\left\{1,2-m,2-(\alpha+\ell),\frac{n}{2}(1-m)+\left(\frac{n}{2}-1\right)(\alpha+\ell-1)\right\}$ we see $\theta \in (0,1)$. 
Thus we can apply the Gagliardo--Nirenberg inequality to find $c_1>0$ such that
  \begin{align}\label{GN}
    \int_\Omega (u+\delta)^{p+\alpha+\ell-1}\,dx
    &= \|(u+\delta)^{\frac{p+m-1}{2}}\|_{L^{\frac{2(p+\alpha+\ell-1)}{p+m-1}}(\Omega)}
         ^{\frac{2(p+\alpha+\ell-1)}{p+m-1}}
    \\ \notag
    &\le c_1  \|\nabla(u+\delta)^{\frac{p+m-1}{2}}\|_{L^2(\Omega)}
                  ^{\frac{2(p+\alpha+\ell-1)}{p+m-1}\theta}
                  \|(u+\delta)^{\frac{p+m-1}{2}}\|_{L^{\frac{2}{p+m-1}}(\Omega)}
                  ^{\frac{2(p+\alpha+\ell-1)}{p+m-1}(1-\theta)}
    \\ \notag
    &\quad\,
          + c_1 \|(u+\delta)^{\frac{p+m-1}{2}}\|_{L^{\frac{2}{p+m-1}}(\Omega)}
                  ^{\frac{2(p+\alpha+\ell-1)}{p+m-1}}
  \end{align}
for all $t \in (0,\tmax)$.
Moreover, thanks to \eqref{case1}, we have
  \[
    \frac{2(p+\alpha+\ell-1)}{p+m-1}\theta
    =\frac{p+\alpha+\ell-2}{\frac{1}{2}\left(p+m-2+\frac{2}{n}\right)}<2.
  \]
Hence, noticing from Lemma \ref{L1} that $\int_\Omega u\,dx \le M_*$,
from \eqref{GN} and Young's inequality we can take $c_2>0$ such that
  \begin{align}\label{term2-1}
    \frac{p(p-1)C_SL}{p+\alpha-1} \int_\Omega (u+\delta)^{p+\alpha+\ell-1}\,dx
    \le \frac{2p(p-1)C_D}{(p+m-1)^2} \int_\Omega |\nabla (u+\delta)^{\frac{p+m-1}{2}}|^2\,dx
    +c_2
  \end{align}
for all $t \in (0,\tmax)$. 
A combination of \eqref{comesti} and \eqref{term2-1} yields that
  \[
    \frac{d}{dt} \int_\Omega (u+\delta)^p\,dx 
    \le \widetilde{C}_\ep \int_\Omega (u+\delta)^{p}\,dx 
         -\frac{p\mu}{2(1+\ep)} \int_\Omega (u+\delta)^{p+\kappa-1}\,dx + c_2
  \]
for all $t \in (0,\tmax)$. By H\"{o}lder's inequality there exists $c_3>0$ such that
  \[
    \frac{d}{dt} \int_\Omega (u+\delta)^p\,dx 
    \le \widetilde{C}_\ep \int_\Omega (u+\delta)^{p}\,dx 
         -c_3 \left(\int_\Omega (u+\delta)^p\,dx\right)^{\frac{p+\kappa-1}{p}} + c_2
  \]
for all $t \in (0,\tmax)$, which yields \eqref{lemLp-1} by an ODE comparison argument.
\end{proof}

\subsection{Case $\mathbf{2}$. $\alpha+\ell<\kappa$ and $\mu>0$.}

In this subsection we show \eqref{Lp} under the condition that $\alpha+\ell<\kappa$ and $\mu>0$.
%
%
\begin{lem}\label{GBlem2}
Let $\mu>0$ and assume that $\alpha>0$, $\kappa>1$ and $\ell>0$ satisfy 
\begin{align}\label{case2}
  \alpha+\ell<\kappa.
\end{align}
Then for any $p>1$ there exists $C>0$ such that 
\begin{align}\label{lemLp-2}
  \int_\Omega u^p(x,t)\,dx \le C \quad \mbox{for all}\ t \in (0,\tmax).
\end{align}
\end{lem}
\begin{proof}
We know that there exist $\ep>0$ and $\widetilde{C}_\ep>0$ such that \eqref{comesti} holds.
By virtue of \eqref{case2}, we have
  \[
    p+\alpha+\ell-1<p+\kappa-1.
  \]
Thus,  by using Young's inequality, we can find $c_1>0$ such that 
  \begin{align}\label{term2-2}
    \frac{p(p-1)C_SL}{p+\alpha-1} \int_\Omega (u+\delta)^{p+\alpha+\ell-1}\,dx
    \le \frac{p\mu}{4(1+\ep)} \int_\Omega (u+\delta)^{p+\kappa-1}\,dx + c_1
  \end{align}
for all $t \in (0,\tmax)$. Combining \eqref{term2-2} with \eqref{comesti} 
and applying H\"{o}lder's inequality, we observe that there exists $c_2>0$ such that
  \[
    \frac{d}{dt} \int_\Omega (u+\delta)^p\,dx 
    \le \widetilde{C}_\ep \int_\Omega (u+\delta)^{p}\,dx 
         -c_2 \left(\int_\Omega (u+\delta)^p\,dx\right)^{\frac{p+\kappa-1}{p}} + c_1
  \]
for all $t \in (0,\tmax)$. Accordingly, we see that \eqref{lemLp-2} holds.
\end{proof}
\subsection{Case $\mathbf{3}$. $\alpha+\ell=\kappa$ and $\mu>\frac{n(\alpha+\ell-m)-2}{2(\alpha-1)+n(\alpha+\ell-m)}C_SL$.}

In order to prove \eqref{Lp} under the condition that $\alpha+\ell=\kappa$ and $\mu>\frac{n(\alpha+\ell-m)-2}{2(\alpha-1)+n(\alpha+\ell-m)}C_SL$,
we first derive the $L^p$-estimate for some $p < 1+\frac{\alpha \mu}{(C_SL-\mu)_{+}}$.
%
%
\begin{lem}\label{someLp}
Let $\mu>0$ and assume that $\alpha>0$, $\kappa>1$ and $\ell>0$ satisfy $\alpha+\ell=\kappa$. Then for any $p \in \left(1,1+\frac{\alpha \mu}{(C_SL-\mu)_{+}}\right)$ there exists $C>0$ such that 
\[
  \int_\Omega u^p(x,t)\,dx \le C \quad \mbox{for all}\ t \in (0,\tmax).
\]
\end{lem}
\begin{proof}
Since the condition $p < 1+\frac{\alpha \mu}{(C_SL-\mu)_{+}}$ implies that
  \[
    \frac{p(p-1)C_SL}{p+\alpha-1} - p \mu<0,
  \]
we can take $\ep>0$ small enough such that
  \[
    \frac{p(p-1)C_SL}{p+\alpha-1} - \frac{p \mu}{1+\ep}<0.
  \]
Thus we have that there exists $\widetilde{C}_\ep>0$ such that
  \begin{align*}
    \frac{d}{dt} \int_\Omega (u+\delta)^p\,dx 
    &\le -\frac{4p(p-1)C_D}{(p+m-1)^2} \int_\Omega |\nabla (u+\delta)^{\frac{p+m-1}{2}}|^2\,dx 
           +\widetilde{C}_\ep \int_\Omega (u+\delta)^{p}\,dx
    \\ \notag
    &\quad\,
           -\left(\frac{p\mu}{1+\ep} - \frac{p(p-1)C_SL}{p+\alpha-1} \right) 
             \int_\Omega (u+\delta)^{p+\kappa-1}\,dx
  \end{align*}
for all $t \in (0,\tmax)$. By using H\"{o}lder's inequality, we obtain $c_1>0$ such that
  \[
    \frac{d}{dt} \int_\Omega (u+\delta)^p\,dx 
    \le \widetilde{C}_\ep \int_\Omega (u+\delta)^{p}\,dx 
         -c_1 \left(\int_\Omega (u+\delta)^p\,dx\right)^{\frac{p+\kappa-1}{p}}
  \]
for all $t \in (0,\tmax)$, and thereby we can arrive at the conclusion.
\end{proof}

Next we establish the $L^p$-estimate for any $p>1$.
%
%
\begin{lem}\label{GBlem3}
Assume that $m\in\mathbb{R}$, $\alpha>0$, $\mu>0$, $\kappa>1$ and $\ell>0$ satisfy 
  \begin{align}\label{case3}
    \alpha+\ell=\kappa
    \quad\mbox{and}\quad
    \mu>\frac{n(\alpha+\ell-m)-2}{2(\alpha-1)+n(\alpha+\ell-m)}C_SL.
  \end{align}
Then for any $p>1$ there exists $C>0$ such that 
\begin{align}\label{lemLp-3}
  \int_\Omega u^p(x,t)\,dx \le C \quad \mbox{for all}\ t \in (0,\tmax).
\end{align}
\end{lem}
\begin{proof}
The second condition of \eqref{case3} yields that
  \[
    \left(1+\frac{\alpha \mu}{(C_SL-\mu)_{+}}\right) - \frac{n}{2}(\alpha+\ell-m) >0.
  \]
Therefore we can pick some 
$p_0 \in \left(\frac{n}{2}(\alpha+\ell-m),1+\frac{\alpha \mu}{(C_SL-\mu)_{+}}\right)$.
Thanks to Lemma \ref{someLp}, we see that there exists $c_1>0$ such that
  \begin{align}\label{Lp0}
    \int_\Omega u^{p_0}\,dx \le c_1
  \end{align}
for all $t \in (0,\tmax)$. Moreover, we choose 
  \[
    p>\max\left\{p_0,p_0+1-m,p_0+1-(\alpha+\ell),\frac{n}{2}(1-m)+\left(\frac{n}{2}-1\right)(\alpha+\ell-1)\right\}
  \]
and take $\ep>0$ and $\widetilde{C}_\ep>0$ such that \eqref{comesti} holds.
Applying the Gagliardo--Nirenberg inequality, we have
  \begin{align*}
    \int_\Omega (u+\delta)^{p+\alpha+\ell-1}\,dx
    &=\|(u+\delta)^\frac{p+m-1}{2}\|_{L^{\frac{2(p+\alpha+\ell-1)}{p+m-1}}(\Omega)}
        ^{\frac{2(p+\alpha+\ell-1)}{p+m-1}}
    \\ \notag
    &\le c_2\|\nabla(u+\delta)^\frac{p+m-1}{2}\|_{L^2(\Omega)}
                ^{\frac{2(p+\alpha+\ell-1)}{p+m-1}\tilde{\theta}}
                \|(u+\delta)^\frac{p+m-1}{2}\|_{L^{\frac{2p_0}{p+m-1}}(\Omega)}
                ^{\frac{2(p+\alpha+\ell-1)}{p+m-1}(1-\tilde{\theta})}
    \\ \notag
    &\quad\,
           + c_2\|(u+\delta)^\frac{p+m-1}{2}\|_{L^{\frac{2p_0}{p+m-1}}(\Omega)}
                  ^{\frac{2(p+\alpha+\ell-1)}{p+m-1}}
  \end{align*}
for all $t \in (0,\tmax)$ with some $c_2>0$, where
  \[
    \tilde{\theta}:=\frac{\frac{p+m-1}{2p_0}-\frac{p+m-1}{2(p+\alpha+\ell-1)}}
                                {\frac{p+m-1}{2p_0}+\frac{1}{n}-\frac{1}{2}} 
    \in (0,1).
  \]
Here, we note from $p_0>\frac{n}{2}(\alpha+\ell-m)$ that
  \begin{align*}
    \frac{2(p+\alpha+\ell-1)}{p+m-1}\tilde{\theta} - 2
    =\frac{\frac{p+\alpha+\ell-1}{p_0}-1-\left(\frac{p+m-1}{p_0}+\frac{2}{n}-1\right)}
               {\frac{p+m-1}{2p_0}+\frac{1}{n}-\frac{1}{2}}
    =\frac{\frac{\alpha+\ell-m}{p_0}-\frac{2}{n}}
               {\frac{p+m-1}{2p_0}+\frac{1}{n}-\frac{1}{2}} < 0.
  \end{align*}
Thus, due to \eqref{Lp0} and Young's inequality, we can show that there is $c_3>0$ such that
  \begin{align}\label{term2-3}
    \frac{p(p-1)C_SL}{p+\alpha-1} \int_\Omega (u+\delta)^{p+\alpha+\ell-1}\,dx
    \le \frac{2p(p-1)C_D}{(p+m-1)^2} \int_\Omega |\nabla (u+\delta)^{\frac{p+m-1}{2}}|^2\,dx
    +c_3
  \end{align}
for all $t \in (0,\tmax)$. From \eqref{comesti} and \eqref{term2-3} we infer that 
  \[
    \frac{d}{dt} \int_\Omega (u+\delta)^p\,dx 
    \le \widetilde{C}_\ep \int_\Omega (u+\delta)^{p}\,dx 
         -\frac{p\mu}{1+\ep} \int_\Omega (u+\delta)^{p+\kappa-1}\,dx + c_3
  \]
for all $t \in (0,\tmax)$, which implies that \eqref{lemLp-3} holds.
\end{proof}
\subsection{Proof of Theorem \ref{thm1}}
In this subsection we complete the proof of boundedness.
\begin{proof}[Proof of Theorem \ref{thm1}]
Thanks to \eqref{thm1condi1} and \eqref{thm1condi2}, we can apply Lemmas \ref{GBlem1}, \ref{GBlem2} and \ref{GBlem3}.
Therefore, for any $p>1$ we can find $c_1>0$ such that 
  \[
    \int_\Omega u^p\,dx \le c_1 
  \]
for all $t \in (0,\tmax)$.
By the Moser iteration (see \cite[Lemma A.1]{T-W}), we obtain
  \[
    \|u(\cdot,t)\|_{L^\infty(\Omega)} \le \infty
  \]
for all $t \in (0,\tmax)$, which concludes the proof.
\end{proof}

\section{Finite-time blow-up}\label{blow-up}
In this section we will show Theorem \ref{thm2}. 
In the following let $\Omega:=B_R(0) \subset \mathbb{R}^n\ (n\ge1)$ be a ball with some $R>0$ and let $\lambda>0$, $\mu>0$ and $\kappa>1$.
Also, we suppose that $D$, $S$ and $f$ fulfill \eqref{Dclass}, \eqref{Sclass} and \eqref{fclass}, respectively, and
$u_0$ satisfies \eqref{u0}.
Moreover, introducing $r:=|x|$, we denote by $(u,v)=(u(r,t),v(r,t))$ the radially symmetric local solution of \eqref{JL} on $[0,\tmax)$. 
Based on \cite{J-L}, we define the mass accumulation function $w$ such that
  \begin{align}\label{defw}
    w(s,t):=\int^{s^\frac{1}{n}}_{0} \rho^{n-1}u(\rho,t)\,d\rho
    \quad\mbox{for}\ s \in [0,R^n]\ \mbox{and}\ t \in [0,\tmax).
  \end{align}
This implies that 
  \[
    w_s(s,t)=\frac{1}{n}u(s^\frac{1}{n},t)
    \quad \mbox{and} \quad
    w_{ss}(s,t)=\frac{1}{n^2}s^{\frac{1}{n}-1}u_r(s^\frac{1}{n},t)
  \]
for all $s \in (0,R^n)$ and $t \in (0,\tmax)$.
Thus we have from the first equation in \eqref{JL} that
  \begin{align}\label{firsteq}
    w_t = n^2s^{2-\frac{2}{n}}D(nw_s)w_{ss}
            - s^{1-\frac{1}{n}}S(nw_s)v_r
            + \lambda w - n^{\kappa-1}\mu \int^{s}_{0} w_s^\kappa(\sigma,t)\,d\sigma
  \end{align}
for all $s \in (0,R^n)$ and $t \in (0,\tmax)$, and see from the second equation in \eqref{JL} that
  \begin{align}\label{secondeq}
    s^{1-\frac{1}{n}}v_r = \overline{M_f}(t)\frac{s}{n} 
                                  - \frac{1}{n}\int^{s}_{0}f(nw_s(\sigma,t))\,d\sigma
  \end{align}
for all $s \in (0,R^n)$ and $t \in (0,\tmax)$. 
From \eqref{firsteq} and \eqref{secondeq} it follows that
  \begin{align}\label{wt}
    w_t &\ge n^2s^{2-\frac{2}{n}}D(nw_s)w_{ss}
            - \frac{1}{n} sS(nw_s)\overline{M_f}(t)
    \\ \notag
    &\quad\,
            + \frac{1}{n} S(nw_s)\int^{s}_{0}f(nw_s(\sigma,t))\,d\sigma
            - n^{\kappa-1}\mu \int^{s}_{0} w_s^\kappa(\sigma,t)\,d\sigma
  \end{align}
for all $s \in (0,R^n)$ and $t \in (0,\tmax)$. 

\medskip

In Subsection \ref{pre} we recall some lemmas in order to obtain inequalities for a derivative of a moment-type functional. In Subsection \ref{estimates} we establish some estimates which lead to differential inequalities for the moment-type functional. The proof of Theorem \ref{blowup} is shown in Subsection \ref{blowup}.
Finally, we give open problems in Subsection \ref{open}.
%
\subsection{Preliminaries}\label{pre}

We first derive the concavity of $w$.
%
%
\begin{lem}\label{con_w}
Assume that $u_0$ satisfies \eqref{u0}. Then
  \[
    u_r(r,t) \le 0 \quad \mbox{for all}\ r \in (0,R) \ \mbox{and}\ t \in (0,\tmax),
  \]
that is, for $w$ as in \eqref{defw}
  \[
    w_{ss}(s,t) \le 0 \quad \mbox{for all}\ s \in (0,R^n) \ \mbox{and}\ t \in (0,\tmax).
  \]
\end{lem}
\begin{proof}
By an argument similar to that in the proof of \cite[Lemma 2.2]{W-2018_nonlinear} or \cite[Lemma 5.1]{B-F-L}, we can prove this lemma.
\end{proof}

Given $s_0 \in (0,R^n)$ and $\gamma \in (-\infty,1)$, we set the moment-type functional 
  \[
    \phi(t):=\int^{s_0}_{0} s^{-\gamma}(s_0-s)w(s,t)\,ds
    \quad\mbox{for}\ t \in [0,\tmax).
  \]
Moreover, we introduce the functional
  \[
    \psi(t):=\int^{s_0}_{0} s^{1-\gamma}(s_0-s)w_s^{\alpha+\ell}(s,t)\,ds
    \quad\mbox{for}\ t \in (0,\tmax)
  \]
and
  \begin{align*}
    S_{\phi}:=\left\{ t \in (0,\tmax) \;\middle|\; \phi(t) \ge \frac{M-s_0}{(1-\gamma)(2-\gamma)\omega_n} \cdot s_0^{2-\gamma} \right\}.
  \end{align*}
Next we recall the following two lemmas which were showed in \cite{W-2018_nonlinear}.
%
%
\begin{lem}\label{westi}
Assume that $u_0$ satisfies \eqref{u0} and let $s_0 \in (0,R^n)$ and $\gamma \in (-\infty,1)$. Then 
  \[
    w\left(\frac{s_0}{2},t\right) 
    \le \frac{1}{\omega_n} \cdot \left(M-\frac{4s_0}{2^\gamma(3-\gamma)}\right)
    \quad\mbox{for all}\ t \in S_{\phi}.
  \]
\end{lem}
%
The following lemma is obtained from Lemmas \ref{con_w} and \ref{westi} 
(see \cite[Lemma 3.2]{W-2018_nonlinear}).
%
%
\begin{lem}
Assume that $u_0$ satisfies \eqref{u0} and let $s_0 \in \left(0,\frac{R^n}{4}\right]$ and $\gamma \in (-\infty,1)$.
Then
  \begin{align}\label{olM}
    \overline{M_f}(t) 
    \le f_\gamma + \frac{1}{2s} \int^{s}_{0} f(nw_s(\sigma,t))\,d\sigma
    \quad\mbox{for all}\ s \in (0,s_0) \ \mbox{and}\ t \in S_{\phi},
  \end{align}
where 
  \begin{align}\label{olC}
    f_\gamma:=f\left(\frac{8n}{2^\gamma(3-\gamma)\omega_n}\right)>0.
  \end{align}
\end{lem}
%
In order to derive differential inequalities for $\phi$ we establish an estimate for $\phi'$.
%
%
\begin{lem}\label{lemphi}
Assume that $f$ fulfills \eqref{fineq2} and $u_0$ satisfies \eqref{u0}. Let $s_0 \in \left(0,\frac{R^n}{4}\right]$ and $\gamma \in (-\infty,1)$ as well as
  \begin{align}\label{gammacondi}
    \gamma<2-\frac{2}{n}.
  \end{align}
Then 
  \begin{align}\label{phiineq}
    \phi'(t) 
    &\ge \frac{n^{\ell-1}}{2} L 
           \int^{s_0}_{0} s^{1-\gamma}(s_0-s)S(nw_s(s,t))w_s^\ell(s,t)\,ds
    \\ \notag
    &\quad\,
         - \frac{f_\gamma}{n}
            \int^{s_0}_{0} s^{1-\gamma}(s_0-s)S(nw_s(s,t))\,ds
    \\ \notag
    &\quad\,
         + n^2 \int^{s_0}_{0} s^{2-\frac{2}{n}-\gamma}(s_0-s)D(nw_s(s,t))w_{ss}(s,t)\,ds
    \\ \notag
    &\quad\,
         - n^{\kappa-1} \mu 
            \int^{s_0}_{0} s^{-\gamma}(s_0-s)
            \left\{ \int^{s}_{0} w_s^\kappa(\sigma,t)\,d\sigma \right\}\,ds
    \\ \notag
    &=: I_1 + I_2 + I_3 + I_4
  \end{align}
for all $t \in S_{\phi}$, where $f_\gamma>0$ is defined as \eqref{olC}.
\end{lem}
\begin{proof}
Invoking \eqref{wt} and \eqref{olM}, we have
  \begin{align}\label{wtolM}
    w_t &\ge n^2s^{2-\frac{2}{n}}D(nw_s)w_{ss}
                  - \frac{f_\gamma}{n} sS(nw_s)
    \\ \notag
    &\quad\,
                  + \frac{1}{2n}S(nw_s) \int^{s}_{0} f(nw_s(\sigma,t))\,d\sigma
            - n^{\kappa-1}\mu \int^{s}_{0} w_s^\kappa(\sigma,t)\,d\sigma
  \end{align}
for all $s \in \left(0,\frac{R^n}{4}\right]$ and $t \in S_{\phi}$.
Here, we note from Lemma \ref{con_w} that 
  \[
    w_s(\sigma,t) \ge w_s(s,t)\quad (\sigma \le s).
  \]
Thanks to this inequality and \eqref{fineq2}, we see that
  \begin{align}\label{wsl}
    S(sw_s) \int^{s}_{0} f(nw_s(\sigma,t))\,d\sigma
    \ge L S(nw_s) \int^{s}_{0} (nw_s(\sigma,t))^\ell\,d\sigma
    \ge n^\ell L sS(nw_s)w_s^\ell
  \end{align}
for all $s \in \left(0,\frac{R^n}{4}\right]$ and $t \in S_{\phi}$.
By virtue of \eqref{wtolM} and \eqref{wsl}, we attain \eqref{phiineq}.
\end{proof}
%
\subsection{Estimates for the four integrals in the inequality \eqref{phiineq}}\label{estimates}

In this subsection, in order to derive different inequalities for $\phi$ we show estimates for the four integrals in \eqref{phiineq} by using lower bound for $\psi$.
We first provide the estimate for $I_1+I_2$ in the following lemma.
%
%
\begin{lem}
Assume that $S$ and $f$ fulfill \eqref{DSineq2} and \eqref{fineq2}, and $u_0$ satisfies \eqref{u0}. Let $s_0 \in \left(0,\frac{R^n}{4}\right]$ and $\gamma \in (-\infty,1)$.
Suppose that $\alpha>0$ and $\ell>0$ satisfy 
  \begin{align}\label{alcondi}
    \alpha+\ell>1.
  \end{align}
Then there exists 
$C_1=C_1(\alpha,\ell,L,C_S)>0$ and $C_2=C_2(R,\alpha,\ell,L,\gamma)>0$ such that
  \begin{align}\label{I1I2ineq}
    I_1+I_2 \ge C_1 \psi(t) - C_2 s_0^{3-\gamma}
  \end{align}
for all $t \in S_{\phi}$.
\end{lem}
\begin{proof}
We define the function $\chi_{A}$ as the characteristic function of the set $A$ and 
put 
  \[
    \overline{C}:=\left(\frac{4f_\gamma}{L}\right)^\frac{1}{\ell}>0.
  \]
As to $I_2$, noticing that $S$ is nondecreasing, we see that
  \begin{align}\label{I2ineq1}
    I_2 &= - \frac{f_\gamma}{n} \int^{s_0}_{0} 
                  \chi_{\{nw_s(\cdot,t) \ge \overline{C}\}}
                  s^{1-\gamma}(s_0-s)S(nw_s)\,ds
    \\ \notag
    &\quad\,
           - \frac{f_\gamma}{n} \int^{s_0}_{0} 
                \chi_{\{nw_s(\cdot,t) < \overline{C}\}}
                s^{1-\gamma}(s_0-s)S(nw_s)\,ds
    \\ \notag
         &\ge - \frac{f_\gamma}{n} \int^{s_0}_{0} 
                      \chi_{\{nw_s(\cdot,t) \ge \overline{C}\}}
                      s^{1-\gamma}(s_0-s)S(nw_s)\,ds
    \\ \notag
    &\quad\,
                 - \frac{f_\gamma}{n}S(\overline{C}) \int^{s_0}_{0} 
                      \chi_{\{nw_s(\cdot,t) < \overline{C}\}}
                      s^{1-\gamma}(s_0-s)\,ds
  \end{align}
for all $t \in S_{\phi}$.
Moreover, we have that
  \begin{align}\label{I2ineq2}
    &- \frac{f_\gamma}{n} \int^{s_0}_{0} 
           \chi_{\{nw_s(\cdot,t) \ge \overline{C}\}}
           s^{1-\gamma}(s_0-s)S(nw_s)\,ds
    \\ \notag
    &\quad\,
    \ge - \frac{f_\gamma}{n} \int^{s_0}_{0} 
               \chi_{\{nw_s(\cdot,t) \ge \overline{C}\}}
               s^{1-\gamma}(s_0-s)S(nw_s) \left(\frac{nw_s}{\overline{C}}\right)^\ell\,ds
    \\ \notag
    &\quad\,
    = - \frac{n^{\ell-1}}{4}L \int^{s_0}_{0} s^{1-\gamma}(s_0-s)S(nw_s) w_s^\ell\,ds
  \end{align}
and
  \begin{align}\label{I2ineq3}
    - \frac{f_\gamma}{n}S(\overline{C}) \int^{s_0}_{0} 
                      \chi_{\{nw_s(\cdot,t) < \overline{C}\}}
                      s^{1-\gamma}(s_0-s)\,ds
    \ge - \frac{f_\gamma S(\overline{C})}{(2-\gamma)(3-\gamma)n} s_0^{3-\gamma}
  \end{align}
for all $t \in S_{\phi}$.
In light of \eqref{I2ineq1}--\eqref{I2ineq3}, we observe that
  \begin{align}\label{I1I2alpha}
    I_1+I_2 
    \ge \frac{n^{\ell-1}}{4} L \int^{s_0}_{0} s^{1-\gamma}(s_0-s)S(nw_s)w_s^\ell\,ds
          - \frac{f_\gamma S(\overline{C})}{(2-\gamma)(3-\gamma)n} s_0^{3-\gamma}
  \end{align}
for all $t \in S_{\phi}$.
Recalling \eqref{DSineq2}, we can obtain
  \begin{align}\label{Salpha}
    \int^{s_0}_{0} s^{1-\gamma}(s_0-s)S(nw_s)w_s^\ell\,ds
    \ge nC_S\int^{s_0}_{0} s^{1-\gamma}(s_0-s)(nw_s+\delta)^{\alpha-1}w_s^{\ell+1}\,ds
  \end{align}
for all $t \in S_\phi$.
If $\alpha \ge 1$, then it follows from 
$(nw_s+\delta)^{\alpha-1}\ge(nw_s)^{\alpha-1}$ that
  \begin{align}\label{alpha1}
    nC_S\int^{s_0}_{0} s^{1-\gamma}(s_0-s)(nw_s+\delta)^{\alpha-1}w_s^{\ell+1}\,ds
    \ge n^\alpha C_S
          \psi(t)
  \end{align}
for all $t \in S_\phi$. Hence, in the case $\alpha \ge 1$ a combination of \eqref{I1I2alpha}, \eqref{Salpha} and \eqref{alpha1} yields \eqref{I1I2ineq}.
On the other hand, if $\alpha<1$, 
then we can show from $w_s^{\ell+1}=\frac{1}{n}w_s^\ell(nw_s+\delta-\delta)$ that
  \begin{align}\label{alpha2}
    &nC_S\int^{s_0}_{0} s^{1-\gamma}(s_0-s)(nw_s+\delta)^{\alpha-1}w_s^{\ell+1}\,ds
    \\ \notag
    &\quad\,
    = nC_S\int^{s_0}_{0} \chi_{\{nw_s(\cdot,t) \ge \delta\}}
                  s^{1-\gamma}(s_0-s)(nw_s+\delta)^{\alpha-1}w_s^{\ell+1}\,ds
    \\ \notag
    &\quad\,\quad\,
        + nC_S\int^{s_0}_{0} \chi_{\{nw_s(\cdot,t) < \delta\}}
                  s^{1-\gamma}(s_0-s)(nw_s+\delta)^{\alpha-1}w_s^{\ell+1}\,ds
    \\ \notag
    &\quad\,
    \ge \frac{n^\alpha}{2^{1-\alpha}} 
           C_S\int^{s_0}_{0} \chi_{\{nw_s(\cdot,t) \ge \delta\}}
                  s^{1-\gamma}(s_0-s)w_s^{\alpha+\ell}\,ds
    \\ \notag
    &\quad\,\quad\,
         + C_S\int^{s_0}_{0} \chi_{\{nw_s(\cdot,t) < \delta\}}
                  s^{1-\gamma}(s_0-s)(nw_s+\delta)^{\alpha}w_s^{\ell}\,ds
    \\ \notag
    &\quad\,\quad\,
         - \delta C_S\int^{s_0}_{0} \chi_{\{nw_s(\cdot,t) < \delta\}}
                  s^{1-\gamma}(s_0-s)(nw_s+\delta)^{\alpha-1}w_s^{\ell}\,ds
    \\ \notag
    &\quad\,
    \ge \frac{n^\alpha}{2^{1-\alpha}} 
           C_S\int^{s_0}_{0} \chi_{\{nw_s(\cdot,t) \ge \delta\}}
                  s^{1-\gamma}(s_0-s)w_s^{\alpha+\ell}\,ds
    \\ \notag
    &\quad\,\quad\,
         + n^\alpha C_S\int^{s_0}_{0} \chi_{\{nw_s(\cdot,t) < \delta\}}
                  s^{1-\gamma}(s_0-s)w_s^{\alpha+\ell}\,ds
    \\ \notag
    &\quad\,\quad\,
         - \delta C_S\int^{s_0}_{0} \chi_{\{nw_s(\cdot,t) < \delta\}}
                  s^{1-\gamma}(s_0-s)(nw_s+\delta)^{\alpha-1}w_s^{\ell}\,ds
  \end{align}
for all $t \in S_\phi$. Noting from $\alpha<1$ that
  \[
    (nw_s+\delta)^{\alpha-1}w_s^\ell 
    = \left(\frac{nw_s}{nw_s+\delta}\right)^{1-\alpha} n^{\alpha-1} w_s^{\alpha+\ell-1}
    \le n^{\alpha-1} w_s^{\alpha+\ell-1},
  \]
we establish that
  \begin{align*}
    &- \delta C_S\int^{s_0}_{0} \chi_{\{nw_s(\cdot,t) < \delta\}}
         s^{1-\gamma}(s_0-s)(nw_s+\delta)^{\alpha-1}w_s^{\ell}\,ds
    \\ \notag
    &\quad\,
    \ge - n^{\alpha-1} C_S \int^{s_0}_{0} \chi_{\{nw_s(\cdot,t) < \delta\}}
                  s^{1-\gamma}(s_0-s)w_s^{\alpha+\ell-1}\,ds
    \\ \notag
    &\quad\,
    \ge - n^{\alpha-1} C_S \int^{s_0}_{0} \chi_{\{nw_s(\cdot,t) < \delta\}}
                  s^{1-\gamma}(s_0-s)\,ds
    \\ \notag
    &\quad\,
    \ge - \frac{n^{\alpha-1} C_S}{(2-\gamma)(3-\gamma)} s_0^{3-\gamma}
  \end{align*}
for all $t \in S_\phi$.
From this inequality and \eqref{alpha2} we see that
  \begin{align}\label{alpha3}
    nC_S\int^{s_0}_{0} s^{1-\gamma}(s_0-s)(nw_s+\delta)^{\alpha-1}w_s^{\ell+1}\,ds
    \ge \frac{n^\alpha}{2^{1-\alpha}} 
           C_S
           \psi(t)
          - \frac{n^{\alpha-1} C_S}{(2-\gamma)(3-\gamma)} s_0^{3-\gamma}
  \end{align}
for all $t \in S_\phi$. Thus, in the case $\alpha<1$, from \eqref{I1I2alpha}, \eqref{Salpha} and \eqref{alpha3} we attain \eqref{I1I2ineq}.
\end{proof}

Next, we show the estimate for $I_3$.
%
%
\begin{lem}
Assume that $D$ fulfills \eqref{DSineq2} and $u_0$ satisfies \eqref{u0}. Let $s_0 \in \left(0,\frac{R^n}{4}\right]$.
Suppose that $m \in \mathbb{R}$, $\alpha>0$, $\ell>0$ and $\gamma \in (-\infty,1)$ satisfy 
  \begin{alignat}{4}
    \label{mcondi1}
    &\mbox{if}\ m \ge 0, &\quad\mbox{then}\quad &\alpha + \ell > m &\quad &\mbox{and} &\quad 2-\frac{2}{n}\cdot\frac{\alpha+\ell}{\alpha+\ell-m} &> \gamma,
  \\ \label{mcondi2}
    &\mbox{if}\ m < 0, &\mbox{then}\quad & & & & \quad 2-\frac{2}{n}  &> \gamma.
  \end{alignat}
Then there exist $\ep>0$, $C_1=C_1(m,\alpha,\ell,\delta,\gamma,C_D)>0$ 
and  $C_2=C_2(m,\delta,\gamma,C_D)>0$ such that
  \begin{align}\label{I3ineq}
    I_3 \ge
    \begin{cases}
      -C_1 s_0^{(3-\gamma)\frac{\alpha+\ell-m}{\alpha+\ell}-\frac{2}{n}} 
            \psi^{\frac{m}{\alpha+\ell}}(t)
      - C_2 s_0^{3-\gamma-\frac{2}{n}} \quad &\mbox{if}\ m>0,
    \\ 
      -C_1 s_0^{(3-\gamma)\frac{\alpha+\ell-\ep}{\alpha+\ell}-\frac{2}{n}} 
            \psi^{\frac{\ep}{\alpha+\ell}}(t)
      - C_2 s_0^{3-\gamma-\frac{2}{n}} \quad &\mbox{if}\ m=0, 
    \\
      -C_2 s_0^{3-\gamma-\frac{2}{n}} \quad &\mbox{if}\ m<0
    \end{cases}
  \end{align}
for all $t \in S_\phi$.
\end{lem}
%
%
\begin{remark}
In this lemma, the constants $C_1>0$ and $C_2>0$ are depend on $\delta$.
However, in the case $m>0$, we can take them which are independent of $\delta$. 
\end{remark}
\begin{proof}
We have from \eqref{DSineq2} that
  \begin{align*}
    I_3 
    &\ge n^2 C_D \int^{s_0}_{0} s^{2-\frac{2}{n}-\gamma}(s_0-s)
            (nw_s+\delta)^{m-1}w_{ss}\,ds
    \\ \notag
    &=  n C_D \int^{s_0}_{0} s^{2-\frac{2}{n}-\gamma}(s_0-s)
          \frac{d}{ds}\left\{\int^{nw_s}_{0}(\xi+\delta)^{m-1}\,d\xi\right\}\,ds
  \end{align*}
for all $t \in S_\phi$.
Since it follows that
  \begin{align*}
    \int^{nw_s}_{0}(\xi+\delta)^{m-1}\,d\xi
    \le
    \begin{cases}
      \dfrac{1}{m}(nw_s+\delta)^m \quad&\mbox{if}\ m>0,
    \\[2mm] \notag
      \log(nw_s+\delta) - \log \delta \quad&\mbox{if}\ m=0,
    \\[2mm] \notag
      -\dfrac{1}{m}\delta^m \quad&\mbox{if}\ m<0,
    \end{cases}
  \end{align*}
we obtain from integrating by parts that
  \begin{align}\label{I3estimate}
    I_3 \ge
    \begin{cases}
      - \dfrac{n}{m} C_D \left(2-\dfrac{2}{n}-\gamma\right) 
         \displaystyle \int^{s_0}_{0} s^{1-\frac{2}{n}-\gamma}(s_0-s)(nw_s+\delta)^m\,ds
         \quad&\mbox{if}\ m>0,
    \\[4mm] 
      - n C_D \left(2-\dfrac{2}{n}-\gamma\right) 
         \displaystyle \int^{s_0}_{0} s^{1-\frac{2}{n}-\gamma}(s_0-s)\log \left(\dfrac{nw_s}{\delta}+1\right)\,ds
         \quad&\mbox{if}\ m=0,
    \\[4mm] 
      \dfrac{n}{m} \delta^m C_D \left(2-\dfrac{2}{n}-\gamma\right) 
      \displaystyle \int^{s_0}_{0} s^{1-\frac{2}{n}-\gamma}(s_0-s)\,ds
      \quad&\mbox{if}\ m<0
    \end{cases}
  \end{align}
for all $t \in S_\phi$.
First, we show the estimate \eqref{I3ineq} in the case $m>0$.
By applying the inequality $(nw_s+\delta)^m \le 2^m((nw_s)^m + \delta^m)$,
we know that
  \begin{align}\label{J1J2}
    \int^{s_0}_{0} s^{1-\frac{2}{n}-\gamma}(s_0-s)(nw_s+\delta)^m\,ds
    &\le
      2^mn^m \int^{s_0}_{0} s^{1-\frac{2}{n}-\gamma}(s_0-s)w_s^m\,ds
    \\ \notag
    &\quad\,
      + 2^m\delta^m \int^{s_0}_{0} s^{1-\frac{2}{n}-\gamma}(s_0-s)\,ds 
    \\ \notag
    &=:J_1 + J_2 
  \end{align}
for all $t \in S_\phi$. 
Invoking from \eqref{mcondi1} that $\frac{m}{\alpha+\ell}<1$, 
we see from H\"{o}lder's inequality that
  \begin{align*}
    J_1 &= 2^mn^m \int^{s_0}_{0} 
            \left[s^{1-\gamma}(s_0-s)w_s^{\alpha+\ell}\right]^\frac{m}{\alpha+\ell}
            \cdot
            s^{(1-\gamma)\frac{\alpha+\ell-m}{\alpha+\ell}-\frac{2}{n}}
            (s_0-s)^\frac{\alpha+\ell-m}{\alpha+\ell}\,ds
    \\ \notag
    &\le 2^m n^m \psi^\frac{m}{\alpha+\ell}(t)
           \cdot
           \left(\int^{s_0}_{0} 
             s^{1-\gamma-\frac{2}{n}\cdot\frac{\alpha+\ell}{\alpha+\ell-m}}(s_0-s)\,ds
           \right)^\frac{\alpha+\ell-m}{\alpha+\ell}
  \end{align*}
for all $t \in S_\phi$. 
Moreover, thanks to the condition $2-\frac{2}{n}\cdot\frac{\alpha+\ell}{\alpha+\ell-m} > \gamma$, we can observe
  \[
    \int^{s_0}_{0} s^{1-\gamma-\frac{2}{n}\cdot\frac{\alpha+\ell}{\alpha+\ell-m}}(s_0-s)\,ds
    =c_1 s_0^{3-\gamma-\frac{2}{n}\cdot\frac{\alpha+\ell}{\alpha+\ell-m}},
  \]
where 
  \[
    c_1:=\frac{1}{\left(2-\gamma-\frac{2}{n}\cdot\frac{\alpha+\ell}{\alpha+\ell-m}\right) \left(3-\gamma-\frac{2}{n}\cdot\frac{\alpha+\ell}{\alpha+\ell-m}\right)}>0.
  \]
Thus we establish that
  \begin{align}\label{J1esti}
    J_1 \le 2^m n^m c_1^\frac{\alpha+\ell-m}{\alpha+\ell}
               s_0^{(3-\gamma)\frac{\alpha+\ell-m}{\alpha+\ell}-\frac{2}{n}} 
               \psi^\frac{m}{\alpha+\ell}(t) 
  \end{align}
for all $t \in S_\phi$.
Also, since 
$2-\gamma-\frac{2}{n} > 2-\gamma-\frac{2}{n}\cdot\frac{\alpha+\ell}{\alpha+\ell-m}>0$ 
and $\delta\le1$,
it follows that
  \begin{align}\label{J2esti}
    J_2 =
    \frac{2^m \delta^m}{\left(2-\gamma-\frac{2}{n}\right) \left(3-\gamma-\frac{2}{n}\right)}
    s_0^{2-\gamma-\frac{2}{n}}
    \le 
    \frac{2^m}{\left(2-\gamma-\frac{2}{n}\right) \left(3-\gamma-\frac{2}{n}\right)}
    s_0^{2-\gamma-\frac{2}{n}}.
  \end{align}
In the case $m>0$, from \eqref{I3estimate}--\eqref{J2esti} we can deduce that
  \[
    I_3 \ge - \frac{2^m n^{m+1} C_D}{m} \left(2-\frac{2}{n}-\gamma\right)
                 c_1^\frac{\alpha+\ell-m}{\alpha+\ell}
                 s_0^{(3-\gamma)\frac{\alpha+\ell-m}{\alpha+\ell}-\frac{2}{n}} 
                 \psi^\frac{m}{\alpha+\ell}(t) 
               - \frac{2^m n C_D}{m \left(3-\gamma-\frac{2}{n}\right)}
                  s_0^{2-\gamma-\frac{2}{n}} 
  \]
for all $t \in S_\phi$, which implies \eqref{I3ineq}.
Next, we confirm that the estimate \eqref{I3ineq} holds in the case $m=0$.
Due to \eqref{mcondi1} with $m=0$, we can take $\ep>0$ small enough such that
  \[
    \alpha+\ell > \ep 
    \quad\mbox{and}\quad
     2-\frac{2}{n}\cdot\frac{\alpha+\ell}{\alpha+\ell-\ep} > \gamma.
  \]
Furthermore, we have that
  \[
    \log \left(\frac{nw_s}{\delta}+1\right) 
    \le \frac{1}{\ep}\left(\frac{nw_s}{\delta}+1\right)^\ep - \frac{1}{\ep}
    = \frac{1}{\ep \delta^\ep}(nw_s+\delta)^\ep - \frac{1}{\ep}.
  \]
In light of \eqref{I3estimate}, we obtain that
  \begin{align}\label{I3estim=01}
    I_3 &\ge 
    - \frac{n C_D}{\ep \delta^\ep} \left(2-\frac{2}{n}-\gamma\right) 
       \int^{s_0}_{0} s^{1-\frac{2}{n}-\gamma}(s_0-s)(nw_s+\delta)^\ep\,ds
    \\ \notag
    &\quad\,
    -  \frac{n C_D}{\ep} \left(2-\frac{2}{n}-\gamma\right) 
      \int^{s_0}_{0} s^{1-\frac{2}{n}-\gamma}(s_0-s)\,ds
  \end{align}
for all $t \in S_\phi$.
As in the case $m>0$, we can verify that
  \begin{align}\label{I3estim=02}
    &- \frac{n C_D}{\ep \delta^\ep} \left(2-\frac{2}{n}-\gamma\right) 
       \int^{s_0}_{0} s^{1-\frac{2}{n}-\gamma}(s_0-s)(nw_s+\delta)^\ep\,ds
    \\ \notag
    &\ge
      - \frac{2^\ep n^{\ep+1} C_D}{\ep \delta^\ep} \left(2-\frac{2}{n}-\gamma\right)
         c_2^\frac{\alpha+\ell-\ep}{\alpha+\ell}
         s_0^{(3-\gamma)\frac{\alpha+\ell-\ep}{\alpha+\ell}-\frac{2}{n}} 
         \psi^\frac{\ep}{\alpha+\ell}(t) 
      - \frac{2^\ep n C_D}{\ep \left(3-\gamma-\frac{2}{n}\right)}
         s_0^{2-\gamma-\frac{2}{n}} 
  \end{align}
for all $t \in S_\phi$, where 
  \[
    c_2:=\frac{1}{\left(2-\gamma-\frac{2}{n}\cdot\frac{\alpha+\ell}{\alpha+\ell-\ep}\right) \left(3-\gamma-\frac{2}{n}\cdot\frac{\alpha+\ell}{\alpha+\ell-\ep}\right)}>0.
  \]
On the other hand, we see that
  \begin{align}\label{I3estim=03}
    -  \frac{n C_D}{\ep} \left(2-\frac{2}{n}-\gamma\right) 
      \int^{s_0}_{0} s^{1-\frac{2}{n}-\gamma}(s_0-s)\,ds
    = - \frac{n C_D}{\ep \left(3-\gamma-\frac{2}{n}\right)}
       s_0^{3-\gamma-\frac{2}{n}}.
  \end{align}
Accordingly, a combination of \eqref{I3estim=01}--\eqref{I3estim=03} yields \eqref{I3ineq}.
Finally, in the case $m<0$, we can show from \eqref{I3estimate} that
  \[
    \frac{n}{m} \delta^m C_D \left(2-\frac{2}{n}-\gamma\right) 
    \int^{s_0}_{0} s^{1-\frac{2}{n}-\gamma}(s_0-s)\,ds
    = \frac{n \delta^m C_D}{m \left(3-\gamma-\frac{2}{n}\right)} 
       s_0^{3-\gamma-\frac{2}{n}},
  \]
which concludes the proof.
\end{proof}

In the following lemma we derive the estimate for $I_4$.
%
%
\begin{lem}\label{lemI4}
Assume that $u_0$ satisfies \eqref{u0}. Let $s_0 \in \left(0,\frac{R^n}{4}\right]$.
Suppose that $\alpha>0$, $\kappa>1$, $\ell>0$ and $\gamma \in (-\infty,1)$ fulfill
  \begin{align}\label{alkcondi}
    \alpha+\ell>\kappa
    \quad\mbox{and}\quad
    2-\frac{\alpha+\ell}{\kappa}<\gamma<1.
  \end{align} 
Then there exists $C_1=C_1(\alpha,\mu,\kappa,\ell,\gamma)>0$ such that
  \begin{align}\label{I4ineq}
    I_4 \ge - C_1s_0^{(3-\gamma)\frac{\alpha+\ell-\kappa}{\alpha+\ell}}
    \psi^\frac{\kappa}{\alpha+\ell}(t)
  \end{align}
for all $t \in S_\phi$.
\end{lem}
\begin{proof}
We apply the Fubini theorem to obtain that
  \begin{align*}
       \int^{s_0}_{0} s^{-\gamma}(s_0-s)
       \left\{ \int^{s}_{0} w_s^\kappa(\sigma,t)\,d\sigma \right\}\,ds
    &= \int^{s_0}_{0} 
       \left\{ \int^{s_0}_{\sigma} s^{-\gamma}(s_0-s) \,ds \right\}
       w_s^\kappa(\sigma,t)\,d\sigma
    \\ \notag
    &\le \frac{1}{1-\gamma} s_0^{1-\gamma}
           \int^{s_0}_{0} (s_0-\sigma)w_s^\kappa(\sigma,t)\,d\sigma
  \end{align*}
for all $t \in S_\phi$.
Thus we have that
  \begin{align}\label{I4esti1}
    I_4 \ge - \frac{n^{\kappa-1}\mu}{1-\gamma} s_0^{1-\gamma}
                 \int^{s_0}_{0} (s_0-s)w_s^\kappa\,ds
  \end{align}
for all $t \in S_\phi$. 
Owing to the first condition of \eqref{alkcondi}, 
we see from H\"{o}lder's inequality that
  \begin{align}\label{I4esti2}
    \int^{s_0}_{0} (s_0-s)w_s^\kappa\,ds
    &= \int^{s_0}_{0} 
       \left[s^{1-\gamma}(s_0-s)w_s^{\alpha+\ell}\right]^\frac{\kappa}{\alpha+\ell}
       \cdot
       s^{-(1-\gamma)\frac{\kappa}{\alpha+\ell}}
       (s_0-s)^\frac{\alpha+\ell-\kappa}{\alpha+\ell}\,ds
    \\ \notag
    &\le \psi^\frac{\kappa}{\alpha+\ell}(t)
           \cdot
           \left(\int^{s_0}_{0} s^{-(1-\gamma)\frac{\kappa}{\alpha+\ell-\kappa}}
                  (s_0-s)\,ds\right)^\frac{\alpha+\ell-\kappa}{\alpha+\ell}
  \end{align}
for all $t \in S_\phi$. 
Here, noting from the second condition of \eqref{alkcondi} that
  \[
    1-(1-\gamma)\frac{\kappa}{\alpha+\ell-\kappa}
    >1-\left(\frac{\alpha+\ell}{\kappa}-1\right)\frac{\kappa}{\alpha+\ell-\kappa}
    =0,
  \]
we can verify that
  \begin{align}\label{I4esti3}
    \int^{s_0}_{0} s^{-(1-\gamma)\frac{\kappa}{\alpha+\ell-\kappa}}(s_0-s)\,ds
    =c_1 s_0^{2-(1-\gamma)\frac{\kappa}{\alpha+\ell-\kappa}},
  \end{align}
where 
  \[
    c_1:=\frac{1}{\left(1-(1-\gamma)\frac{\kappa}{\alpha+\ell-\kappa}\right)
                      \left(2-(1-\gamma)\frac{\kappa}{\alpha+\ell-\kappa}\right)}>0.
  \]
Thanks to \eqref{I4esti1}--\eqref{I4esti3}, it follows that
  \begin{align*}
    I_4 \ge - \frac{n^{\kappa-1}\mu}{1-\gamma}c_1^\frac{\alpha+\ell-\kappa}{\alpha+\ell}
                 s_0^{1-\gamma+\frac{2(\alpha+\ell-\kappa)}
                                         {\alpha+\ell}-(1-\gamma)\frac{\kappa}{\alpha+\ell}}
                 \psi^\frac{\kappa}{\alpha+\ell}(t)
         = - \frac{n^{\kappa-1}\mu}{1-\gamma}c_1^\frac{\alpha+\ell-\kappa}{\alpha+\ell}
              s_0^{(3-\gamma)\frac{\alpha+\ell-\kappa}{\alpha+\ell}}
              \psi^\frac{\kappa}{\alpha+\ell}(t) 
  \end{align*}
for all $t \in S_\phi$, which implies \eqref{I4ineq}.
\end{proof}

In the next lemma we establish the estimate for $w$ which is used later.
%
%
\begin{lem}\label{wpsi}
Assume that $u_0$ satisfies \eqref{u0}. Let $s_0 \in \left(0,\frac{R^n}{4}\right]$.
Suppose that $\alpha>0$, $\ell>0$ and $\gamma \in (-\infty,1)$ fulfill
  \begin{align}\label{wcondi}
    \alpha+\ell>1
    \quad\mbox{and}\quad
    2-(\alpha+\ell)<\gamma<1.
  \end{align} 
Then there exists $C_1=C_1(\alpha,\ell,\gamma)>0$ such that
  \[
    w(s,t) \le C_1 s^\frac{\alpha+\ell+\gamma-2}{\alpha+\ell}(s_0-s)^{-\frac{1}{\alpha+\ell}}
                  \psi^\frac{1}{\alpha+\ell}(t)
  \]
for all $s \in (0,s_0)$ and $t \in S_\phi$.
\end{lem}
\begin{proof}
According to the condition $\alpha+\ell>1$, we have from H\"{o}lder's inequality that
  \begin{align*}
    w(s,t) 
    &= \int^{s}_{0} w_s(\sigma,t)\,d\sigma
    \\ \notag
    &= \int^{s}_{0} [\sigma^{1-\gamma}(s_0-\sigma)]^\frac{1}{\alpha+\ell}w_s(\sigma,t)
         \cdot
         [\sigma^{1-\gamma}(s_0-\sigma)]^{-\frac{1}{\alpha+\ell}}\,d\sigma
    \\ \notag
    &\le \psi^\frac{1}{\alpha+\ell}(t)
           \cdot
           \left(\int^{s}_{0} \sigma^{-\frac{1-\gamma}{\alpha+\ell-1}}
                                  (s_0-\sigma)^{-\frac{1}{\alpha+\ell-1}}\,d\sigma 
           \right)^\frac{\alpha+\ell-1}{\alpha+\ell}
  \end{align*}
for all $s \in (0,s_0)$ and $t \in S_\phi$.
Moreover, thanks to the condition $2-(\alpha+\ell)<\gamma<1$, we see that
  \begin{align*}
    \int^{s}_{0} \sigma^{-\frac{1-\gamma}{\alpha+\ell-1}}
                   (s_0-\sigma)^{-\frac{1}{\alpha+\ell-1}}\,d\sigma 
    &\le (s_0-s)^{-\frac{1}{\alpha+\ell-1}}
         \int^{s}_{0} \sigma^{-\frac{1-\gamma}{\alpha+\ell-1}}\,d\sigma 
    \\ \notag
    &= \left(\frac{\alpha+\ell-1}{\alpha+\ell+\gamma-2}\right)
         s^\frac{\alpha+\ell+\gamma-2}{\alpha+\ell-1} (s_0-s)^{-\frac{1}{\alpha+\ell-1}}.
  \end{align*}
Thus we can obtain that
  \[
    w(s,t) \le
    \left(\frac{\alpha+\ell-1}{\alpha+\ell+\gamma-2}\right)^\frac{\alpha+\ell-1}{\alpha+\ell}
    s^\frac{\alpha+\ell+\gamma-2}{\alpha+\ell} (s_0-s)^{-\frac{1}{\alpha+\ell}}
    \psi^\frac{1}{\alpha+\ell}(t)
  \]
for all $s \in (0,s_0)$ and $t \in S_\phi$, which concludes the proof.
\end{proof}

From Lemma \ref{wpsi} we drive the estimate for $\psi$.
%
%
\begin{lem}\label{psiphi}
Assume that $u_0$ satisfies \eqref{u0}. Let $s_0 \in \left(0,\frac{R^n}{4}\right]$.
Suppose that $\alpha>0$, $\ell>0$ and $\gamma \in (-\infty,1)$ fulfill
  \[
    \alpha+\ell>1
    \quad\mbox{and}\quad
    2-(\alpha+\ell)<\gamma<1.
  \] 
Then there exists $C_1=C_1(\alpha,\ell,\gamma)>0$ such that
  \begin{align}\label{psiphiineq}
    \psi(t) \ge C_1 s_0^{-(3-\gamma)(\alpha+\ell-1)}\phi^{\alpha+\ell}(t)
  \end{align}
for all $t \in S_\phi$.
\end{lem}
\begin{proof}
By an argument similar to that in the proof of \cite[Lemma 3.7]{Wang-Li}, 
we can show that \eqref{psiphiineq} holds.
\end{proof}
%
\subsection{ODIs for $\phi$. Proof of Theorem \ref{thm2}}\label{blowup}
In this subsection we will prove Theorem \ref{thm2}.
To this end, we first derive the ODIs for the moment-type functional $\phi$ in the following lemma.
%
%
\begin{lem}\label{lemodi}
Assume that $D$, $S$ and $f$ fulfill \eqref{DSineq2} and \eqref{fineq2}. 
Suppose that $m\in\mathbb{R}$, $\alpha>0$, $\kappa>1$ and $\ell>0$ satisfy that
  \begin{align}
    \label{malkcondi1}
    &\mbox{if}\ m\ge0, \quad\mbox{then}\quad \alpha+\ell>\max\left\{m+\frac{2}{n}\kappa,\kappa\right\},
  \\ \label{malkcondi2}
    &\mbox{if}\ m<0, \quad\mbox{then}\quad \alpha+\ell>\max\left\{\frac{2}{n}\kappa,\kappa\right\}.
  \end{align}
Then there exist $\ep>0$ small enough and one can find $\gamma=\gamma(m,\alpha,\kappa,\ell) \in (-\infty,1)$ and
$C=C(R,m,\alpha,\mu,\kappa,\ell,L,\delta,\gamma,C_D,C_S)>0$ such that
if $u_0$ satisfies \eqref{u0} and $s_0 \in \left(0,\frac{R^n}{4}\right]$, then
  \begin{align}\label{odiphi}
    \phi'(t) \ge 
    \begin{cases}
      \dfrac{1}{C}s_0^{-(3-\gamma)(\alpha+\ell-1)}\phi^{\alpha+\ell}(t) 
      -Cs_0^{3-\gamma-\frac{2}{n}\cdot\frac{\alpha+\ell}{\alpha+\ell-m}}
      \quad&\mbox{if}\ m>0,
    \\[4mm]
      \dfrac{1}{C}s_0^{-(3-\gamma)(\alpha+\ell-1)}\phi^{\alpha+\ell}(t) 
      -Cs_0^{3-\gamma-\frac{2}{n}\cdot\frac{\alpha+\ell}{\alpha+\ell-\ep}}
      \quad&\mbox{if}\ m=0,
    \\[4mm]
      \dfrac{1}{C}s_0^{-(3-\gamma)(\alpha+\ell-1)}\phi^{\alpha+\ell}(t) 
      -Cs_0^{3-\gamma-\frac{2}{n}}
      \quad&\mbox{if}\ m<0
    \end{cases}
  \end{align}
for all $t \in S_\phi$.
\end{lem}
\begin{proof}
By virtue of \eqref{malkcondi1}, it follows that if $m\ge0$, then
  \begin{align}\label{takeg}
    \left(2-\frac{2}{n}\cdot\frac{\alpha+\ell}{\alpha+\ell-m}\right)
    - \left(2-\frac{\alpha+\ell}{\kappa}\right)
    &= (\alpha+\ell)\left(\frac{1}{\kappa}-\frac{2}{n}\cdot\frac{1}{\alpha+\ell-m}\right)
    \\ \notag
   &> (\alpha+\ell)\left(\frac{1}{\kappa}-\frac{2}{n}\cdot\frac{n}{2\kappa}\right)=0.
  \end{align}
Thus, in the case $m\ge1$ we can find $\gamma\in(-\infty,1)$ such that
  \begin{align}\label{malkg1}
    2-\frac{\alpha+\ell}{\kappa}<\gamma<2-\frac{2}{n}\cdot\frac{\alpha+\ell}{\alpha+\ell-m}.
  \end{align}
Thanks to \eqref{malkcondi1} and \eqref{malkg1}, we know that 
\eqref{gammacondi}, \eqref{alcondi}, \eqref{mcondi1}, \eqref{alkcondi} and \eqref{wcondi} hold.
In the case $m>0$, applying Lemmas \ref{lemphi}--\ref{lemI4}, we see that there exist 
$c_1=c_1(\alpha,\ell,L,C_S)>0$ and 
$c_2=c_2(R,m,\alpha,\mu,\kappa,\ell,L,\delta,\gamma,C_D,C_S)>0$
such that
  \begin{align}\label{phi1}
    \phi'(t) &\ge
    c_1\psi(t) - c_2 s_0^{3-\gamma}
    - c_2 s_0^{(3-\gamma)\frac{\alpha+\ell-m}{\alpha+\ell}-\frac{2}{n}} 
           \psi^{\frac{m}{\alpha+\ell}}(t)
    - c_2 s_0^{3-\gamma-\frac{2}{n}}
    \\ \notag
    &\quad\,
    - c_2 s_0^{(3-\gamma)\frac{\alpha+\ell-\kappa}{\alpha+\ell}}
           \psi^\frac{\kappa}{\alpha+\ell}(t)
  \end{align}
for all $t \in S_\phi$.
Noting that $\alpha+\ell>m$ and $\alpha+\ell>\kappa$, from Young's inequality we can take $c_3>0$ and $c_4>0$ such that
  \begin{align*}
    c_2 s_0^{(3-\gamma)\frac{\alpha+\ell-m}{\alpha+\ell}-\frac{2}{n}} 
    \psi^{\frac{m}{\alpha+\ell}}(t)
    &\le \frac{c_1}{4}\psi(t) 
         + c_3s_0^{3-\gamma-\frac{2}{n}\cdot\frac{\alpha+\ell}{\alpha+\ell-m}}
  \intertext{and}
    c_2 s_0^{(3-\gamma)\frac{\alpha+\ell-\kappa}{\alpha+\ell}}
    \psi^\frac{\kappa}{\alpha+\ell}(t)
    &\le \frac{c_1}{4}\psi(t) 
         + c_4s_0^{3-\gamma}.
  \end{align*}
In light of \eqref{phi1}, we obtain that
  \[
    \phi'(t) \ge
    \frac{c_1}{2}\psi(t) 
    - c_2 s_0^{3-\gamma-\frac{2}{n}\cdot\frac{\alpha+\ell}{\alpha+\ell-m}}
    \left(
            s_0^{\frac{2}{n}\cdot\frac{\alpha+\ell}{\alpha+\ell-m}}
            + \frac{c_3}{c_2}
            + s_0^\frac{m}{\alpha+\ell-m}
            + \frac{c_4}{c_2}s_0^{\frac{2}{n}\cdot\frac{\alpha+\ell}{\alpha+\ell-m}}
    \right)   
  \]
for all $t \in S_\phi$. 
By using $s_0\le\frac{R^n}{4}$, we can verify that there exists $c_5>0$ such that
  \[
    \phi'(t) \ge
    \frac{c_1}{2}\psi(t) 
    - c_5 s_0^{3-\gamma-\frac{2}{n}\cdot\frac{\alpha+\ell}{\alpha+\ell-m}}
  \]
for all $t \in S_\phi$.
Moreover, we have from Lemma \ref{psiphi} that there exists $c_6>0$ such that
  \[
    \phi'(t) \ge
    \frac{c_1c_6}{2}s_0^{-(3-\gamma)(\alpha+\ell-1)}\phi^{\alpha+\ell}(t)
    - c_5 s_0^{3-\gamma-\frac{2}{n}\cdot\frac{\alpha+\ell}{\alpha+\ell-m}}
  \]
for all $t \in S_\phi$, which implies \eqref{odiphi} in the case $m>0$.
As to the case $m=0$, due to \eqref{takeg}, we can pick $\ep>0$ small enough and $\gamma\in(-\infty,1)$ such that
  \[
    2-\frac{\alpha+\ell}{\kappa}<\gamma
    <2-\frac{2}{n}\cdot\frac{\alpha+\ell}{\alpha+\ell-\ep}.
  \]
Therefore, using Lemmas \ref{lemphi}--\ref{lemI4}, we establish that there exist 
$c_7=c_7(\alpha,\ell,L,C_S)>0$ and 
$c_8=c_8(R,\alpha,\mu,\kappa,\ell,L,\delta,\gamma,C_D,C_S)>0$ 
such that
  \[
    \phi'(t) \ge
    c_7\psi(t) - c_8 s_0^{3-\gamma}
    - c_8 s_0^{(3-\gamma)\frac{\alpha+\ell-\ep}{\alpha+\ell}-\frac{2}{n}} 
           \psi^{\frac{\ep}{\alpha+\ell}}(t)
    - c_8 s_0^{3-\gamma-\frac{2}{n}}
    - c_8 s_0^{(3-\gamma)\frac{\alpha+\ell-\kappa}{\alpha+\ell}}
           \psi^\frac{\kappa}{\alpha+\ell}(t)
  \]
for all $t \in S_\phi$.
As in the case $m>0$, from this inequality we can attain \eqref{odiphi}.
Finally, in the case $m<0$ we see from \eqref{malkcondi2} that
  \begin{align*}
    \left(2-\frac{2}{n}\right)
    - \left(2-\frac{\alpha+\ell}{\kappa}\right)
    = \frac{\alpha+\ell}{\kappa}-\frac{2}{n}
    > \frac{1}{\kappa}\cdot\frac{2\kappa}{n}-\frac{2}{n}=0.
  \end{align*}
Thus we can take $\gamma\in(-\infty,1)$ satisfying
  \[
    2-\frac{\alpha+\ell}{\kappa}<\gamma<2-\frac{2}{n}.
  \]
By virtue of Lemmas \ref{lemphi}--\ref{lemI4} we know that there exist $c_9=c_9(\alpha,\ell,L,C_S)>0$ and 
$c_{10}=c_{10}(R,m,\alpha,\mu,\kappa,\ell,L,\delta,\gamma,C_D,C_S)>0$ such that
  \[
    \phi'(t) \ge
    c_9\psi(t) - c_{10} s_0^{3-\gamma}
    - c_{10} s_0^{3-\gamma-\frac{2}{n}}
    - c_{10} s_0^{(3-\gamma)\frac{\alpha+\ell-\kappa}{\alpha+\ell}}
           \psi^\frac{\kappa}{\alpha+\ell}(t)
  \]
for all $t \in S_\phi$.
By an argument similar to that in the case $m>0$, we can verify that \eqref{odiphi} holds in the case $m<0$.
\end{proof}

We are in a position to complete the proof of Theorem \ref{thm2}.

\begin{proof}[Proof of Theorem \ref{thm2}]
We first consider the case $m>0$. Due to \eqref{thm2condi1}, we can obtain from Lemma \ref{lemodi} that there exist $\gamma\in(-\infty,1)$, $c_1>0$ and $c_2>0$ such that for each $u_0$ satisfying \eqref{u0} and $s_0\le\frac{R^n}{4}$, it follows that
  \begin{align}\label{firstphiesti}
    \phi'(t)\ge
    c_1s_0^{-(3-\gamma)(\alpha+\ell-1)}\phi^{\alpha+\ell}(t) 
    -c_2s_0^{3-\gamma-\frac{2}{n}\cdot\frac{\alpha+\ell}{\alpha+\ell-m}}
  \end{align}
for all $t\in S_\phi$.
Next we choose $s_0\le\frac{R^n}{4}$ small enough such that
  \begin{align}\label{s0M0}
    s_0 \le \frac{M_0}{2}
  \end{align}
and
  \begin{align}\label{s01}
    s_0^{(\alpha+\ell)\left(1-\frac{2}{n}\cdot\frac{1}{\alpha+\ell-m}\right)}
    \le \frac{c_1}{2c_2}\left(\frac{M_0}{2(1-\gamma)(2-\gamma)\omega_n}\right)
         ^{\alpha+\ell}.
  \end{align}
Furthermore, we fix $\ep_0\in\left(0,\frac{s_0}{2}\right)$ so small and take $s_\star \in (0,s_0)$ fulfilling
  \begin{align}\label{lower}
    \frac{M_0-\ep_0}{\omega_n}\int^{s_0}_{s_\star}s^{-\gamma}(s_0-s)\,ds
    > \frac{M_0-s_0}{(1-\gamma)(2-\gamma)\omega_n}s_0^{2-\gamma}.
  \end{align}
We define $r_\star:=s_\star^\frac{1}{n}\in(0,R)$ and suppose that $u_0$ satisfies \eqref{u0} and \eqref{u0mass}.
In order to show $\tmax<\infty$, assuming that $\tmax=\infty$, we will derive a contradiction.
We set 
  \begin{align}\label{setS}
    \widetilde{S}:=\left\{T\in(0,\infty) \;\middle|\; \phi(t)>\frac{M_0-s_0}{(1-\gamma)(2-\gamma)\omega_n}s_0^{2-\gamma}\quad\mbox{for all}\ t\in[0,T]\right\}.
  \end{align}
Here, we note that $\widetilde{S}$ is not empty.
Indeed, since we have that for any $s\in(s_\star,R^n)$
  \[
    w(s,0)\ge w(s_\star,0)=\frac{1}{\omega_n}\int_{B_{r_\star}(0)}u_0\,dx
    \ge \frac{M_0-\ep_0}{\omega_n},
  \]
we see from \eqref{lower} that
  \begin{align*}
    \phi(0)&\ge \int^{s_0}_{s_\star} s^{-\gamma}(s_0-s)w(s,0)\,ds
  \\
    &\ge \frac{M_0-\ep_0}{\omega_n} \int^{s_0}_{s_\star} s^{-\gamma}(s_0-s)\,ds
  \\
    &> \frac{M_0-s_0}{(1-\gamma)(2-\gamma)\omega_n}s_0^{2-\gamma}.
  \end{align*}
Thus we can put $\widetilde{T}:=\sup \widetilde{S}\in(0,\infty)$.
Moreover, we know that $(0,\widetilde{T})\subset S_\phi$.
Owing to \eqref{setS} and \eqref{s0M0}, we establish that
  \begin{align*}
    \phi(t)\ge\frac{M_0}{2(1-\gamma)(2-\gamma)\omega_n}s_0^{2-\gamma}
  \end{align*}
for all $t\in(0,\widetilde{T})$.
From \eqref{s01} it follows that
  \begin{align*}
    \frac{\frac{c_1}{2}s_0^{-(3-\gamma)(\alpha+\ell-1)}\phi^{\alpha+\ell}(t)}
             {c_2s_0^{3-\gamma-\frac{2}{n}\cdot\frac{\alpha+\ell}{\alpha+\ell-m}}}
    \ge
    \frac{c_1}{2c_2}\left(\frac{M_0}{2(1-\gamma)(2-\gamma)\omega_n}\right)^{\alpha+\ell}
    s_0^{-(\alpha+\ell)+\frac{2}{n}\cdot\frac{\alpha+\ell}{\alpha+\ell-m}}
    \ge1
  \end{align*}
for all $t \in (0,\widetilde{T})$, which implies from \eqref{firstphiesti} that
  \begin{align}\label{key}
    \phi'(t)\ge\frac{c_1}{2}s_0^{-(3-\gamma)(\alpha+\ell-1)}\phi^{\alpha+\ell}(t)\ge0
  \end{align}
for all $t \in (0,\widetilde{T})$.
This inequality yields that $\widetilde{T}=\infty$.
However, from \eqref{key} and $\alpha+\ell-1>0$ we can show that
  \[
    \widetilde{T}\le\frac{2}{(\alpha+\ell-1)c_1}s_0^{(3-\gamma)(\alpha+\ell-1)}.
  \]
As a consequence, we attain that $\tmax$ must be finite.
In the cases $m=0$ and $m<0$, we can prove that $\tmax<\infty$ by an argument similar to that in the case $m>0$.
\end{proof}
%
\subsection{Open problems}\label{open}
In \cite{F_2021_optimal,Wang-Li,W-2018_nonlinear} the critical values such that solutions remain bounded or blow up in finite time were derived.
With regard to the conditions \eqref{thm1condi1}, \eqref{thm2condi1} and \eqref{thm2condi2}, we see that if $n\ge3$ and $m\ge0$ as well as $\frac{n}{n-2}m\le\kappa$, then 
\[
    \max\left\{m+\frac{2}{n},\kappa\right\}
    =\max\left\{m+\frac{2}{n}\kappa,\kappa\right\}=\kappa.
\]
Thus we know that the critical value is $\alpha+\ell=\kappa$ in this case.
However, in the cases that $n\in\{1,2\}$ and that $n\ge3$ and $m\ge0$ as well as $\frac{n}{n-2}m>\kappa$, the conditions \eqref{thm1condi1}, \eqref{thm2condi1} and \eqref{thm2condi2} are not optimal.
Moreover, the special cases are as follows:
\begin{itemize}
\item In the case that $m=\alpha=1$, behavior of solutions is an open problem when $\max\left\{\frac{2}{n},\kappa-1\right\}\le \ell \le\frac{2}{n}\kappa$ (see Figures \ref{kln12fig} and \ref{klfig}).

\begin{figure}[H]\label{kln12}
\begin{minipage}{0.52\columnwidth}
\begin{center}
\hspace*{5mm}
\scalebox{1.03}{\includegraphics{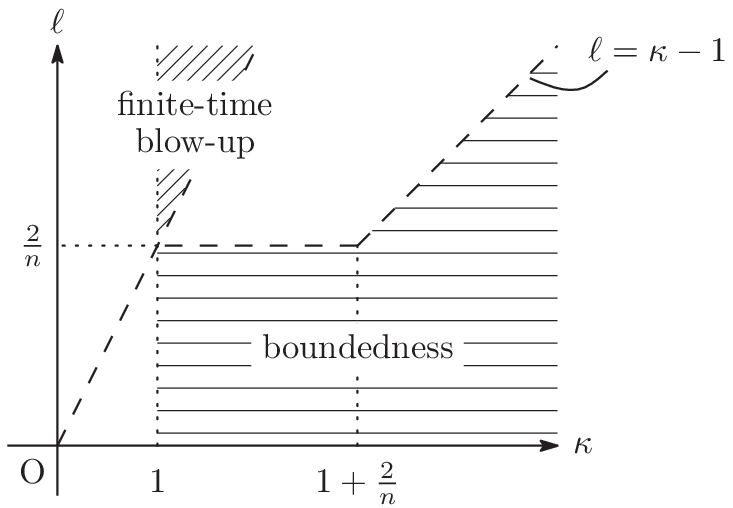}}
\caption{$n\in\{1,2\}$ and $m=\alpha=1$}
\label{kln12fig}
\end{center}
\end{minipage}
\begin{minipage}{0.47\columnwidth}
\begin{center}
\scalebox{1.03}{\includegraphics{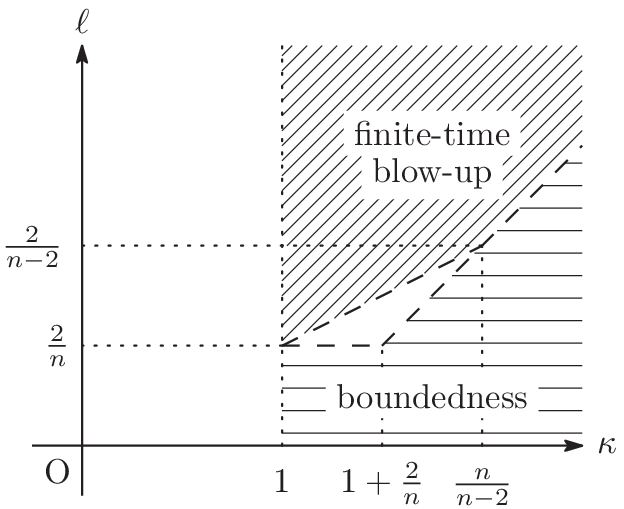}}
\caption{$n\ge3$ and $m=\alpha=1$}
\label{klfig}
\end{center}
\end{minipage}
\end{figure}

\item In the case that $m=1$ and $\kappa<\frac{n}{(n-2)_{+}}$, from Figure \ref{alfig} we have an open question of whether solutions remain bounded or blow up when 
$\max\left\{1+\frac{2}{n},\kappa\right\}\le \alpha+\ell \le 1+\frac{2}{n}\kappa$.

\begin{figure}[H]
\begin{center}
\hspace*{-2.2cm}
\scalebox{1.03}{\includegraphics{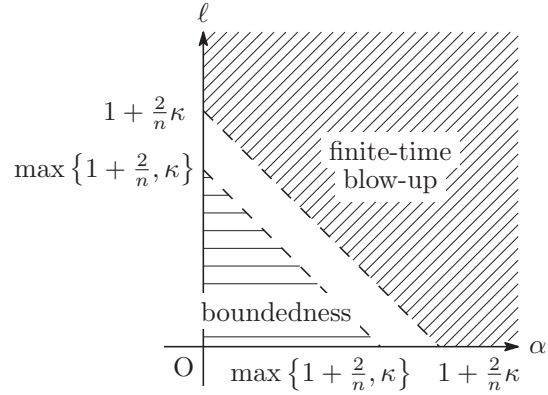}}
\caption{$m=1$ and $\kappa<\frac{n}{(n-2)_{+}}$}
\label{alfig}
\end{center}
\end{figure}

\item In the case that $\alpha=1$ and $\ell>0$, there is an open problem for behavior of solutions
when $n=1$ and 
$\max\{\kappa-1,m+1\} \le \ell \le \max\{2\kappa-1,m+2\kappa-1\}$.
Also, the same question exists when $n\ge2$ and 
$\max\left\{\kappa-1,m-\left(1-\frac{2}{n}\right)\right\} \le \ell \le m-\left(1-\frac{2}{n}\kappa\right)$.
Moreover, in the case that $\alpha>0$ and $\ell=1$, we obtain regions that $\ell$ is replaced by $\alpha$ in Figures \ref{mln1fig} and \ref{mlfig}.
 
\begin{figure}[H]
\hspace*{4mm}
\begin{minipage}{0.45\columnwidth}
\begin{center}
\scalebox{1.03}{\includegraphics{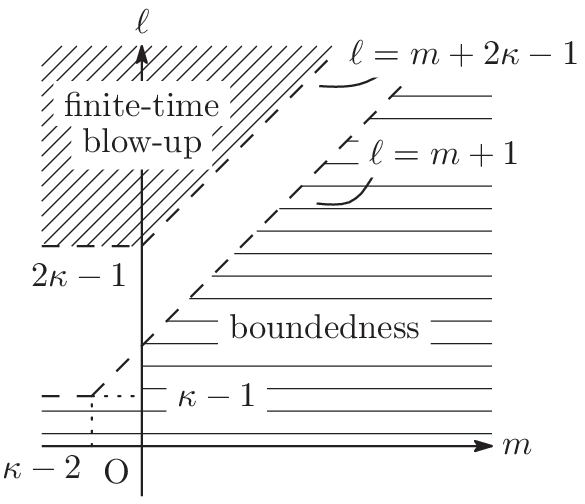}}
\caption{$n=1$ and $\alpha=1$}
\label{mln1fig}
\end{center}
\end{minipage}
\hspace*{-9mm}
\begin{minipage}{0.55\columnwidth}
\begin{center}
\scalebox{1.03}{\includegraphics{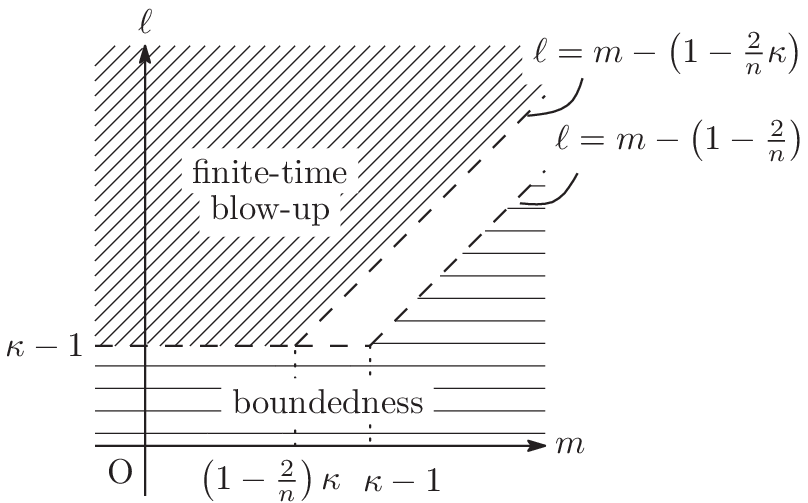}}
\hspace*{-7mm}
\caption{$n\ge2$ and $\alpha=1$}
\label{mlfig}
\end{center}
\end{minipage}
\end{figure}
\end{itemize}

%
\smallskip
\section*{Acknowledgments}
The author would like to thank Professor Tomomi Yokota for his encouragement and helpful comments on the manuscript.
%

\bibliographystyle{plain}

\end{document}